\documentclass[reqno,a4paper]{amsart}
\usepackage{amssymb}
\usepackage{amsmath}
\newcommand{\half}{\frac{1}{2}}
\newcommand{\abs}[1]{\vert #1 \vert}

\newcommand{\R}{\mathbb{R}}

\begin{document} 
\newtheorem{prop}{Proposition}[section]
\newtheorem{Def}{Definition}[section]
\newtheorem{theorem}{Theorem}[section]
\newtheorem{lemma}{Lemma}[section]
 \newtheorem{Cor}{Corollary}[section]

\title[MKG in 2D]{\bf Improved well-posedness results for the
Maxwell-Klein-Gordon system in 2D}
\author[Hartmut Pecher]{
{\bf Hartmut Pecher}\\
Fakult\"at f\"ur  Mathematik und Naturwissenschaften\\
Bergische Universit\"at Wuppertal\\
Gau{\ss}str.  20\\
42119 Wuppertal\\
Germany\\
e-mail {\tt pecher@math.uni-wuppertal.de}}
\date{}

\begin{abstract}
The local well-posedness problem for the Maxwell-Klein-Gordon system in Coulomb gauge as well as Lorenz gauge is treated in two space dimensions for data with minimal regularity assumptions. In the classical case of data in $L^2$-based Sobolev spaces $H^s$ and $H^l$ for the electromagnetic field $\phi$ and the potential $A$, respectively. The minimal regularity assumptions are $s > \half$ and $l > \frac{1}{4}$ , which leaves a gap of $\half$ and $\frac{1}{4}$ to the critical regularity with respect to scaling $s_c = l_c =0$ . This gap can be reduced for data in Fourier-Lebesgue spaces $\widehat{H}^{s,r}$ and $\widehat{H}^{l,r}$ to $s> \frac{21}{16}$  and $l > \frac{9}{8}$ for $r$ close to $1$ , whereas the critical exponents with respect to scaling fulfill $s_c \to 1$ , $ l_c \to 1 $ as $r \to 1$ . Here $\|f\|_{\widehat{H}^{s,r}} := \| \langle \xi \rangle^s \tilde{f}\|_{L^{r'}_{\tau \xi}} \, , \, 1 < r \le 2 \, , \, \frac{1}{r}+\frac{1}{r'} = 1 \, . $ Thus the gap is reduced for $\phi$ as well as $A$ in both gauges.
\end{abstract}
\maketitle
\renewcommand{\thefootnote}{\fnsymbol{footnote}}
\footnotetext{\hspace{-1.5em}{\it 2020 Mathematics Subject Classification:} 
35Q40, 35L70 }
\normalsize 
\setcounter{section}{0}
  
\section{Introduction and main results}
\noindent Consider the Maxwell-Klein-Gordon system 
\begin{align}
\label{1}
D_{\mu} D^{\mu} \phi & = m^2 \phi \\
\label{2}
\partial^{\nu} F_{\mu \nu} & =  j_{\mu} 
 \, ,
\end{align}
where $m>0$ is a constant and
\begin{align}
\label{3}
F_{\mu \nu} & := \partial_{\mu} A_{\nu} - \partial_{\nu} A_{\mu} \\
\label{4}
D_{\mu} \phi & := \partial_{\mu} + iA_{\mu} \phi \\
\nonumber
j_{\mu} & := Im(\phi \overline{D_{\mu} \phi}) = Im(\phi \overline{\partial_{\mu}
\phi}) + |\phi|^2 A_{\mu} \, .
\end{align}
Here $F_{\mu \nu} : {\mathbb R}^{2+1} \to {\mathbb R}$ denotes the
electromagnetic field, $\phi : {\mathbb R}^{2+1} \to {\mathbb C}$ a scalar field
and $A_{\nu} : {\mathbb R}^{2+1} \to {\mathbb R}$ the potential. We use the
notation $\partial_{\mu} = \frac{\partial}{\partial x_{\mu}}$, where we write
$(x^0,x^1,x^2) = (t,x^1,x^2)$ and also $\partial_0 = \partial_t$ and $\nabla =
(\partial_1,\partial_2)$. Roman indices run over $1,2$ and greek indices over
$0,1,2$ and repeated upper/lower indices are summed. Indices are raised and
lowered using the Minkowski metric $diag(-1,1,1)$.

The Maxwell-Klein-Gordon system describes the motion of a spin 0 particle with
mass $m$ self-interacting with an electromagnetic field.

We are interested in local well-posedness (LWP) results, so we have to choose an
appropriate gauge. In this paper we consider the Coulomb gauge $ \partial_j
A^j=0$ and the Lorenz gauge $\partial_{\mu} A^{\mu} =0$ . Our aim is to minimize
the regularity assmptions for the Cauchy data. If we consider data in classical
$L^2$-based $H^s$-spaces we improve the result in Lorenz gauge and also slightly
in Coulomb gauge. In order to improve the results further with respect to
scaling we use data in Fourier-Lebesgue spaces $\widehat{H}^{s,r}$ (cf.  the
definition below) with $1<r \le 2$ and obtain an improvement for $r$ close to 1.

Let us make some historical remarks. In space dimension $n=3$ , Coulomb gauge
and Cauchy data $\phi(0)=\phi_0 \in H^s$ , $(\partial_t \phi)(0) \in H^{s-1}$ ,
$A_j(0)= a_{0j} \in H^l$, $(\partial_t A_j)(0) = b_{0j} \in H^{l-1}$ ,
Klainerman and Machedon \cite{KM} proved global well-posedness in energy space
and above ($s=l \ge 1$) . They detected that the nonlinearities fulfill a null
condition. The global well-posedness result was improved by Keel, Roy and Tao
\cite{KRT} to the condition  $s=l > \frac{\sqrt{3}}{2}$ . LWP for low regularity
data was shown by Cuccagna \cite{C} for $l=s > \frac{3}{4}$ and $n=3$ . Selberg
\cite{S} remarked that a smallness assumption in \cite{C} can be removed.

In the case $n=2$ and with Coulomb gauge Czubak and Pikula \cite{CP} obtained
LWP if either $1\ge s=l > \half$ or $s=\frac{5}{8}+\epsilon $ , $l=
\frac{1}{4}+\epsilon$       , where $\epsilon > 0$ is arbitrary. Their methods are crucial for our results.

Another paper was is fundamental for our results is due to Selberg and Tesfahun \cite{ST}.  In Lorenz gauge and $n=3$ they  proved global
well-posedness for finite energy data $\phi_0 \in H^1$ , $\phi_1 \in L^2$ ,
$F_{\mu \nu}(0)=F^0_{\mu \nu} \in L^2$ , $A_{\nu}(0) = a_{0 \nu} \in \dot{H}^1$
, $(\partial_t A_{\nu})(0) = b_{0 \nu} \in L^2$. The solution fulfills $\phi \in
C^0(\R,H^1) \cap C^1(\R,L^2) $, $F_{\mu \nu} \in C^0(\R,L^2)$ . The potential
possibly loses some regularity compared to the data, which however is of minor
interest. In Lorenz gauge the situation is more delicate, because the
nonlinearity $Im(\phi \overline{\partial_{\mu} \phi})$ has no null structure.
The author \cite{P} obtained LWP for less regular data, namely $\phi_0 \in H^s$
, $\phi_1 \in H^{s-1}$ , $F^0_{\mu \nu} \in H^{s-1}$ , $\nabla a_{0\nu} \in
H^{s-1}$ , $b_{0\nu} \in H^{s-1}$ , with $s >\frac{3}{4}$, where $\phi \in
C^0([0,T],H^s) \cap C^1([0,T],H^{s-1}) $, $F_{\mu \nu} \in C^0([0,T],H^{s-1})
\cap C^1([0,T],H^{s-2}) $ . $A_{\nu}$ loses regularity compared to $a_{0 \nu}$
and $b_{0  \nu}$ , we only obtain $A_{\nu} =  A_{\nu}^{hom} + A_{\nu}^{inh}$
with $\nabla A_{\nu}^{hom} , \partial_t A_{\nu}^{hom} \in C^0([0,T],H^{s-1}) $ ,
$A_{\nu}^{inh} \in C^0([0,T],H^l) \cap C^0([0,T],H^{l-1}) $ , where $l > \half$
. 

In Lorenz gauge and $n=2$ the author \cite{P} proved similar results for
slightly modified solution spaces in the case $s >\frac{3}{4}$ , $l >
\frac{1}{4}$ .

As can be easily seen the minimal regularity for LWP predicted by scaling is
$s_c = \frac{n}{2}-1$ , so that for $n=3$ there is a gap of $\frac{1}{4}$ both
in Coulomb and Lorenz gauge. The author  \cite{P1} for $n=3$ and Lorenz gauge
closed this gap up to the endpoint in the sense of scaling, if Cauchy data are
given in Fourier-Lebesgue spaces $\widehat{H}^{s,r}$ instead of standard
$L^2$-based Sobolev spaces $H^s = \widehat{H}^{s,2}$ . Here we define
$$\|f\|_{\widehat{H}^{s,r}} = \| \langle \xi \rangle^s
\widehat{f}(\xi)\|_{L^{r'}}  \, ,$$ 
 where $\frac{1}{r}+\frac{1}{r'} =1$ , for $1<r \le 2$ . More precisely, given
Cauchy data $\phi_0 \in \widehat{H}^{s,r} $ , $\phi_1 \in \widehat{H}^{s-1,r} $
, $ F^0_{\mu \nu} \in \widehat{H}^{s-1,r}$ , $\nabla a_{0 \nu} \in
\widehat{H}^{s-1,r} $ , $b_{0 \nu} \in \widehat{H}^{s-1,r} $ , where $ s >
\frac{5}{2r} - \half$ , there exists a local solution $ \phi \in
C^0([0,T],\widehat{H}^{s,r}) \cap  C^1([0,T],\widehat{H}^{s-1,r})$, $\nabla
F_{\mu \nu} , \partial_t F_{\mu \nu} \in  C^0([0,T],\widehat{H}^{s-2,r}) $
relative to a potential $A_{\mu} \in C^0([0,T],\widehat{H}^{l,r}) \cap 
C^1([0,T],\widehat{H}^{l-1,r})$, where $l > \frac{3}{r}-1$ . In the limit $r \to
1 $ ($r > 1$) the condition reduces to $ s > 1$ , $l>2$, which is optimal up to
the endpoint.

Cauchy data in Fourier-Lebesgue spaces were previously considered among others
by Gr\"unrock \cite{G} and Gr\"unrock-Vega \cite{GV} for the KdV and the
modified  KdV-equation. They were also used by Gr\"unrock \cite{G1} in order to
prove almost optimal LWP results for wave equations with quadratic nonlinearity
for $n=3$ .  For $n=2$ this was  proven by Grigoryan-Nahmod \cite{GN} for wave
equations with nonlinear terms which fulfill a null condition. These results
rely on an adaptation of bilinear estimates by Foschi-Klainerman \cite{FK} in
the classical $L^2$-case.

In the present paper we consider exclusively the case $n=2$ in Coulomb as well
as Lorenz gauge for data in standard $L^2$-based Sobolev saces as well as
Fourier-Lebesgue spaces. Our object is LWP for data with minimal regularity
assumptions.

Let us first consider the system (\ref{1}),(\ref{2}) in Coulomb gauge. It is
well-known that we obtain the equivalent system
\begin{align}
\label{5}
\Delta A_0 &= -Im(\phi \overline{\partial_t \phi}) + |\phi|^2 A_0 \, , \\
\label{6}
\Box A_j  & =Im(\phi \overline{\partial_j \phi}) +  |\phi|^2 A_j - \partial_j
\partial_t A_0 \, , \\
\label{7}
\Box \phi & = -2i A^j \partial_j \phi + 2i A_0 \partial_t \phi + i (\partial_t
A_0)\phi + A^{\mu}A_{\mu} \phi + m^2 \phi \, , \\
\label{8}
\partial^j A_j & =0 \, .
\end{align}
We consider the Cauchy problem for this system with data
\begin{align}
\label{9}
A_j(0) = a_j \in \widehat{H}^{l,r} \quad &, \quad (\partial_t A_j)(0) = b_j \in
\widehat{H}^{l-1,r} \, , \\
\label{10}
\phi(0) = \phi_0 \in \widehat{H}^{s,r} \quad &, \quad  (\partial_t \phi)(0) =
\phi_1 \in \widehat{H}^{s-1,r} \, , 
\end{align}
which fulfill the compatibility condition
\begin{equation}
\label{11}
\partial^j a_j = \partial_j b_j = 0 \, .
\end{equation}
The natural scaling transformation for Maxwell-Klein-Gordon is given by
$$ A_{\mu}(t,x) \, , \, \phi(t,x) \, \rightarrow \, \lambda A_{\mu}(\lambda
t,\lambda x) \, , \, \lambda \phi(\lambda t, \lambda x) $$
for $\lambda > 0$ . Since $\|\lambda f(\lambda x) \|_{\dot{H}^s} =
\lambda^{s-\frac{n}{2}+1} \|f\|_{\dot{H}^s} $ , we conclude that $s_c =
\frac{n}{2}-1$ is critical, i.e. , LWP is expected for $s \ge s_c$ at best and
ill-posedness for $s<s_c$ . More generally we also obtain for $1<r\le 2$ :
$\|\lambda f(\lambda x) \|_{\dot{\widehat{H}}^{s,r}} = \lambda^{s-\frac{n}{r}+1}
\|f\|_{\dot{\widehat{H}}^{s,r}}  $ , so that $s_c = \frac{n}{r}-1$ . 

Before formulating our main result we introduce some notation.
We denote the Fourier transform with respect to space and time  by
$\,\,\widehat{}\,$ . 
$u \precsim v$ is defined by $|\widehat{u}| \lesssim |\widehat{v}|$ ,
 $\Box = \partial_t^2 - \Delta$ is the d'Alembert operator,
$a\pm := a \pm \epsilon$ for a sufficiently small $\epsilon >0$ , so that $a++ > a+ > a$ , and $\langle
\,\cdot\, \rangle := (1+|\cdot|^2)^{\frac{1}{2}}$ . \\
Let $\Lambda^{\alpha}$
be the multiplier with symbol  $
\langle\xi \rangle^\alpha $ . Similarly let $D^{\alpha}$,
and $D_{-}^{\alpha}$ be the multipliers with symbols $
\abs{\xi}^\alpha$ and $\quad ||\tau|-|\xi||^\alpha$ ,
respectively.

\begin{Def}
\label{Def.}
Let $1\le r\le 2$ , $s,b \in \R$ . The Fourier-Lebesgue space
$\widehat{H}^{s,r}$ is the completion of the Schwarz space ${\mathcal
S}(\R^{2})$ with norm $\|f\|_{\widehat{H}^{s,r}} = \| \langle \xi \rangle^s
\widehat{f}(\xi)\|_{L^{r'}}$  , where $r'$ is the dual exponent to $r$ , and $
\dot{\widehat{H}}^{s,r}$ denotes the homogeneous space.
The wave-Sobolev spaces $H^r_{s,b}$ are the completion of the Schwarz space
${\mathcal S}(\R^{1+2})$ with norm
$$ \|u\|_{H^r_{s,b}} = \| \langle \xi \rangle^s \langle  |\tau| - |\xi|
\rangle^b \widehat{u}(\tau,\xi) \|_{L^{r'}_{\tau \xi}} \, , $$ 
where $r'$ is the dual exponent to $r$.
We also define $H^r_{s,b}[0,T]$ as the space of the restrictions of functions in
$H^r_{s,b}$ to $[0,T] \times \mathbb{R}^2$.  Similarly we define $X^r_{s,b,\pm}
$ with norm  $$ \|\phi\|_{X^r_{s,b\pm}} := \| \langle \xi \rangle^s \langle \tau
\pm |\xi| \rangle^b \tilde{\phi}(\tau,\xi)\|_{L^{r'}_{\tau \xi}} $$ and
$X^r_{s,b,\pm}[0,T] $ .
$\dot{H}^r_{s,b}$ and $\dot{X}^r_{s,b,\pm} $ are the corresponding homogeneous
spaces,  where $\langle \xi \rangle$ is replaced by $|\xi|$ .
In the case $r=2$ we denote $H^2_{s,b}= H^{s,b}$ and $X^2_{s,b\pm} =
X^{s,b}_{\pm}$. For brevity we denote $\|u\|_{X^r_{s,b}} = \|u\|_{X^r_{s,b,+}} +
\|u\|_{X^r_{s,b,-}} $ .
\end{Def}
Our main result is the following theorem.
\begin{theorem}
\label{Theorem1}
Let $1<r \le 2$ . Assume that $s-1 \le l \le s$ , $s >
\frac{13}{8r}-\frac{5}{16}$, $ l > \frac{7}{4r}-\frac{5}{8}$ and $2s-l >
\frac{3}{2r}$ , $2l-s > \frac{2}{r} - \frac{5}{4}$ . Then the
Maxwell-Klein-Gordon system (\ref{5}) - (\ref{11}) is locally
well-posed in the sense that there exist $T > 0$ and $b > \frac{1}{r}$ , such
that there exists a unique solution
$$ A_0 \in C^0([0,T],\dot{H}^{0+}) \cap \dot{H}^{\min{(2s,s+1)}} \cap
C^1([0,T],\dot{H}^{0+} \cap \dot{H}^{\min{(2s-1,s)}}) \, , $$
$$ A_j \in X^r_{l,b,+}[0,T] + X^r_{l,b,-}[0,T]\,  , \, \phi \in 
X^r_{s,b,+}[0,T] + X^r_{s,b,-}[0,T] \, . $$
This solution satisfies
$$ A_j \in C^0([0,T],\widehat{H}^{l,r}) \cap  C^1([0,T],\widehat{H}^{l-1,r}) \,
, \, \phi \in C^0([0,T],\widehat{H}^{s,r}) \cap C^1([0,T],\widehat{H}^{s-1,r})
\, . $$
The solution depends continuously on the data and persistence of higher
regularity pertains.
\end{theorem}
\noindent{\bf Remarks:} 1.  In the case $r=2$ we obtain $s-1 \le l \le s$ , $ s
> \half$ , $l> \frac{1}{4}$ , $2s-l > \frac{1}{4}$, $2l-s > - \frac{1}{4}$ .
This is $\half$ away from the critical value $s_c =0$ . It slightly improves the
result of \cite{CP}. \\
2. In the case $r>1$ , but close to $1$ , the conditions are $ s >
\frac{21}{16}$ , $l> \frac{9}{8}$ , $2s-l >\frac{3}{2}$, $2l-s >\frac{3}{4}$ ,
whereas $s_c \to 1$ as $r \to 1$ . The gap shrinks to $\frac{5}{16}$ . \\
3. It is not at all clear, whether in dimension $n=2$ the gap disappears as in
the case $n=3$ (cf. \cite{P1}) . \\[0.3em]

Now the most important observation was (cf. e.g. \cite{S}) that this system can
be rewritten in the following form, which involves null forms.

\begin{align}
\label{12}
\Delta A_0 &= -Im(\phi \overline{\partial_t \phi}) + |\phi|^2 A_0 \, , \\
\label{13} 
\Delta \partial_t A_0 & = -\partial^j\, Im(\phi \overline{\partial_j \phi}) +
\partial^j (|\phi|^2 A_j) \, , \\
\label{14}
\Box A_j  & = 2 R^k D^{-1} Q_{jk}( Re \, \phi,Im \, \phi) + P(|\phi|^2 A_j) =:
N_j(A_j,\phi) \, , \\
\nonumber
\Box \phi & = -i Q_{jk}(\phi,D^{-1}(R^j A^k-R^k A^j)) + 2i A_0 \partial_t \phi +
i (\partial_t A_0)\phi + A^{\mu}A_{\mu} \phi + m^2 \phi \\
\label{15}
& =: M(A,\phi) \, , \\
\nonumber
\partial^j A_j & =0 \, ,
\end{align}
where $R_k = D^{-1} \partial_k$ is the Riesz transform, $P$ denotes the
projection onto the divergence-free vector fields given by
$PX_j = R^k(R_j X_k - R_k X_j) $
and 
\begin{equation}
\label{q}
Q_{jk}(u,v) = \partial_j u \partial_k v - \partial_k u \partial_j v 
\end{equation}
denotes the null form.

In the classical case $r=2$ we follow the arguments by Czubak-Pikula \cite{CP} ,
especially when handling the elliptic equations. The necessary bilinear
estimates for the null form are by use of Klainerman-Machedon \cite{KM} reduced
to standard bilinear estimates which are contained in the very convenient paper
\cite{AFS}.  The cubic estimates are obtained by using bilinear estimates twice.
In the general case $1 < r \le 2$ , especially for $r$ close to $1$, we handle
the null forms by using estimates from \cite{FK}, as already Gr\"unrock \cite{G}
and Grigoryan-Nahmid \cite{GN} did before. It is also essential to use a result
by Grigoryan-Tanguay \cite{GT} for bilinear estimates in $H^r_{s,b}$-spaces.
Finally, in order to optimize the result in the general case $1<r \le 2$ we
interpolate between the extreme cases $r=1+$ and $r=2$ .

In Lorenz gauge our main result is given in the
following theorem.
\begin{theorem}
\label{Theorem2}
Let $1 < r \le 2$ . Assume that $s-1 \le l \le s$ and $s>
\frac{13}{8r}-\frac{5}{16}$, $ l > \frac{7}{4r} - \frac{5}{8}$ , $2s-l >
\frac{3}{2r}$ , $2l-s >\frac{2}{r} - \frac{5}{4}$ . Let initial data be given
such that 
$A_{\mu}(0)=a_{\mu} \in \widehat{H}^{l,r} \, , \, (\partial A_j)(0) = b_{\mu}
\in \widehat{H}^{l-1,r} \, , \, \phi(0) = \phi_0 \in \widehat{H}^{s,r} \, , \,
(\partial_t \phi)(0) = \phi_0 \in \widehat{H}^{s-1,r}$ , which fulfill the compatibility condition $b_0 =
\partial^j a_j$ . Then there exist $T>0$, $b>\frac{1}{r}$, $b' >
\half+\frac{1}{2r}$ , such that the problem (\ref{1}),(\ref{2}) in Lorenz gauge  $\partial^{\mu} A_{\mu} =0$ has a unique
solution
$$ \phi \in X^r_{s,b,+}[0,T] + X^r_{s,b,-}[0,T] \, , \, A_{\mu} \in
X^r_{l,b',+}[0,T] + X^r_{l,b',-}[0,T] \, . $$
This solution satisfies
$$ \phi \in C^0([0,T],\widehat{H}^{s,r}) \cap C^1([0,T],\widehat{H}^{s-1,r}) \,
, \, A_{\mu} \in C^0([0,T],\widehat{H}^{l,r}) \cap
C^1([0,T],\widehat{H}^{l-1,r}) \, . $$
\end{theorem}
\begin{Cor}
In the classical case $r=2$ assume $s-1 \le l \le s+1$ , $ s > \half$ ,
$l>\frac{1}{4}$ , $2s-l >\frac{3}{4}$ , $2l-s > -\frac{1}{4}$ . Then the system
(\ref{1}),(\ref{2}) in Lorenz gauge  $\partial^{\mu} A_{\mu} =0$  with data $a_{\mu} \in H^l $ , $b_{\mu} \in H^{l-1}$ , $
\phi_0 \in H^s$ , $\phi_1 \in H^{s-1}$  , which fulfill $b_0 = \partial^j a_j$ ,
has a unique local solution
$$\phi \in X^{s,b}_+[0,T] + X^{s,b}_-[0,T] \, , \, A_{\mu} \in X^{l,b'}_+[0,T] +
X^{l,b'}_-[0,T] \, ,$$ 
 where $b > \half$ , $b' > \frac{3}{4}$ .
The solution satisfies
$$\phi \in C^0([0,T],H^s) \cap C^1([0,T],H^{s-1}) \, , \, A_{\mu} \in
C^0([0,T],H^l) \cap C^1([0,T],H^{l-1}) \, . $$
\end{Cor}
{\bf Remarks:}
1. The result of this Corollary improves the former result of the author
\cite{P} from $s> \frac{3}{4}$ to $s> \half$ . This is $\half$ away 
from the scaling critical exponent $s_c =0$ .\\
2. In the case $r=1+$ we have to assume in Theorem \ref{Theorem2} the conditions
$s > \frac{21}{16}$ , $l > \frac{9}{8}$ , $2s-l > \frac{3}{4}$ , $2l-s > -
\frac{1}{4}$ and $s-1 \le l  \le s+1$ . 
Similarly  as in Coulomb gauge this reduces the gap to the scaling critical
exponent $s_c$ , which converges to $1$ as $r \to 1$ . \\[0.3em]

We start by reformulating the system (\ref{1}),(\ref{2}) in Lorenz gauge
\begin{equation}
\label{8'}\partial^{\mu} A_{\mu} =0
\end{equation}
 as follows:
$$\square A_{\mu} = \partial^{\nu} \partial_{\nu} A_{\mu} =
\partial^{\nu}(\partial^{\mu} A_{\nu} - F_{\mu \nu}) = -\partial^{\nu} F_{\mu
\nu} = - j_{\mu} \, , $$
thus (using the notation $\partial = (\partial_0,\partial_1,...,\partial_n)$):
\begin{equation}
\label{6'}
\square A = -Im (\phi \overline{\partial \phi}) - A|\phi|^2 =: N(A,\phi) 
\end{equation}
and
\begin{align*}
m^2 \phi & = D_{\mu} D^{\mu} \phi = \partial_{\mu} \partial^{\mu} \phi -iA_{\mu}
\partial^{\mu} \phi -i\partial_{\mu}(A^{\mu} \phi) - A_{\mu}A^{\mu} \phi \\
& = \square \phi - 2i A^{\mu} \partial_{\mu} \phi - A_{\mu} A^{\mu} \phi \,
\end{align*}
thus
\begin{equation}
\label{7'}
(\square -m^2) \phi = 2i A^{\mu} \partial_{\mu} \phi + A_{\mu} A^{\mu} \phi =:
M(A,\phi) \, .
\end{equation}

In order to prove the theorem we start by recalling the well-known fact that the
term $A^{\mu}  \partial_{\mu}\phi$ has null structure whereas the term $Im(\phi
\overline{\partial \phi})$ seems to fulfill no null condition. The null forms
are then handled similarly as in Coulomb gauge as well as the cubic terms.

\section{Preliminaries}

We start by collecting some fundamental properties of the solution spaces. We
rely on \cite{G}. The spaces $X^r_{s,b,\pm} $ with norm  $$
\|\phi\|_{X^r_{s,b\pm}} := \| \langle \xi \rangle^s \langle \tau \pm |\xi|
\rangle^b \tilde{\phi}(\tau,\xi)\|_{L^{r'}_{\tau \xi}} $$ for $1<r<\infty$ are
Banach spaces with ${\mathcal S}$ as a dense subspace. The dual space is
$X^{r'}_{-s,-b,\pm}$ , where $\frac{1}{r} + \frac{1}{r'} = 1$. The complex
interpolation space is given by
$$(X^{r_0}_{s_0,b_0,\pm} , X^{r_1}_{s_1,b_1,\pm})_{[\theta]} = X^r_{s,b,\pm} \,
, $$
where $s=(1-\theta)s_0+\theta s_1$, $\frac{1}{r} = \frac{1-\theta}{r_0} +
\frac{\theta}{r_1}$ , $b=(1-\theta)b_0 + \theta b_1$ . Similar properties has
the space $H^r_{s,b}$ .\\
If $u=u_++u_-$, where $u_{\pm} \in X^r_{s,b,\pm} [0,T]$ , then $u \in
C^0([0,T],\hat{H}^{s,r})$ , if $b > \frac{1}{r}$ .

The "transfer principle" in the following proposition, which is well-known in
the case $r=2$, also applies for general $1<r<\infty$ (cf. \cite{GN}, Prop. A.2
or \cite{G}, Lemma 1). We denote $ \|u\|_{\hat{L}^p_t(\hat{L}^q_x)} :=
\|\tilde{u}\|_{L^{p'}_{\tau} (L^{q'}_{\xi})}$ .
\begin{prop}
\label{Prop.0.1}
Let $1 \le p,q \le \infty$ .
Assume that $T$ is a bilinear operator which fulfills
$$ \|T(e^{\pm_1 itD} f_1, e^{\pm_2itD} f_2)\|_{\hat{L}^p_t(\hat{L}^q_x)}
\lesssim \|f_1\|_{\hat{H}^{s_1,r}} \|f_2\|_{\hat{H}^{s_2,r}}$$
for all combinations of signs $\pm_1,\pm_2$ , then for $b > \frac{1}{r}$ the
following estimate applies:
$$ \|T(u_1,u_2)\|_{\hat{L}^p_t(\hat{L}^q_x)} \lesssim \|u_1\|_{H^r_{s_1,b}} 
\|u_2\|_{H^r_{s_2,b}} \, . $$
\end{prop}

The general local well-posedness theorem is the following (obvious
generalization of)  \cite{G}, Thm. 1.
\begin{theorem}
\label{Theorem0.3}
Let $N_{\pm}(u,v):=N_{\pm}(u_+,u_-,v_+,v_-)$ and
$M_{\pm}(u,v):=M_{\pm}(u_+,u_-,\\v_+,v_-)$ be multilinear functions.
Assume that for given $s,l \in \R$, $1 < r < \infty$ there exist $ b,b' >
\frac{1}{r}$ such that the estimates
$$ \|N_{\pm}(u,v)\|_{X^r_{s,b-1+,\pm}} \le \omega_1(
\|u\|_{X^r_{s,b}},\|v\|_{X^r_{l,a}}) $$
and 
$$\|M_{\pm}(u,v)\|_{X^r_{l,b'-1+,\pm}} \le \omega_2(
\|u\|_{X^r_{s,b}},\|v\|_{X^r_{l,b'}}) $$
are valid with nondecreasing functions $\omega_j$ , where $\|u\|_{X^r_{s,b}} :=
\|u_-\|_{X^r_{s,b,-}} + \|u_+\|_{X^r_{s,b,+}}$. Then there exist $T=T(\|u_{0_
{\pm}}\|_{\hat{H}^{s,r}},\|v_{0_{\pm}}\|_{\hat{H}^{l,r}})$ $>0$ and a unique
solution $(u_+,u_-,\\v_+,v_-) \in X^r_{s,b,+}[0,T] \times X^r_{s,b,-}[0,T]
\times X^r_{l,b',+}[0,T] \times X^r_{l,b',-}[0,T] $ of the Cauchy problem
$$ \partial_t u_{\pm} \pm i\Lambda u = N_{\pm}(u,v) \quad , \quad \partial_t
v_{\pm} \pm i\Lambda v = M_{\pm}(u,v) $$ $$         u_{\pm}(0) = u_{0_{\pm}} \in
\hat{H}^{s,r} \quad , \quad v_{\pm}(0) = v_{0_{\pm}} \in \hat{H}^{l,r}       \,
. $$
 This solution is persistent and the mapping data upon solution
$(u_{0+},u_{0-},v_{0+},v_{0-}) \\ \mapsto (u_+,u_-,v_+,v_-)$ , $\hat{H}^{s,r}
\times \hat{H}^{s,r}\times \hat{H}^{l,r} \times \hat{H}^{l,r} \to
X^r_{s,b,+}[0,T_0] \times X^r_{s,b,-}[0,T_0] \times X^r_{l,b',+}[0,T_0]\times
X^r_{l,b',-}[0,T_0] $ is locally Lipschitz continuous for any $T_0 < T$.
\end{theorem}

\section{MKG in Coulomb gauge}
In a standard way we rewrite the system (\ref{14}),(\ref{15}) as a first order
(in t) system. Defining
$A_{j,\pm} = \half(A_j \pm (i\Lambda)^{-1} \partial_t A_j) \, , \, \phi_{\pm} =
\half(\phi \pm (i \Lambda)^{-1}\partial_t \phi) $ , so that
$A_j=A_{j,+}+A_{j,-}$ , $\partial_t A_j= i \Lambda(A_{j,+}-A_{j,-})$ ,
$\phi=\phi_++\phi_-$ , $\partial_t \phi = i \Lambda (\phi_+-\phi_-)$ the system
transforms to
\begin{align}
\label{20}
(i \partial_t \pm \Lambda)A_{j,\pm} & = -A_j \mp(2 \Lambda)^{-1} \mathcal
M_j(A_j,\partial_t A_j,\phi,\partial_t \phi)\, ,\\
\label{21'}
(i \partial_t \pm \Lambda)\phi_{\pm} & = -\phi \mp(2 \Lambda)^{-1} \mathcal
N(A_j,\partial_t A_j,\phi,\partial_t \phi) \, .
\end{align}
The initial data transform to 
$$A_{j,\pm}(0)= \half(a_j \pm (i\Lambda)^{-1} b_j) \in 
\widehat{H}^{l,r} \quad , \quad \phi_{\mp}(0)= \half(\phi_0 \pm (i\Lambda)^{-1}
\phi_1) \in \widehat{H}^{s,r} \, .$$ 
This transformation allows to apply Theorem \ref{Theorem0.3}. For the wave
equations one has to estimate
\begin{equation}
\label{17'}
\| \Box A_j \|_{H^r_{s-1,b'-1+}}
\end{equation}
and
\begin{equation}
\label{18}
\| \Box \phi \|_{H^r_{l-1,b-1+}} \, .
\end{equation}
 
We obtain the necessary estimates for (\ref{17'}) in Lemma \ref{Lemma3a} and
Lemma \ref{Lemma4'}, and for (\ref{18})  in Lemma \ref{Lemma1'} , Lemma
\ref{Lemma5'} and Corollary \ref{Cor.2}.
\\[0.7em]

Let us first consider the elliptic equations (\ref{12}) and (\ref{13}). The
equation (\ref{13}) is easier to handle. We prove that it is solved by $B_0 =
\partial_t A_0 \in C^0([0,T],\dot{H}^{\sigma}) $. Defining $A_0(t) := \int_0^t
B_0(s) ds + a_0$ , where $a_0$ is the solution of the following variational
problem at $t=0$:
$$ \int_{\mathbb{R}^2} |\nabla A_0|^2 + |D_0 \phi|^2 \, dx \, \rightarrow \,
\min   $$
in $H^1(\mathbb{R}^2)$ . The following Lemma was shown by Czubak-Pikula
\cite{CP}, Lemma 3.1 and Lemma 3.2.
\begin{lemma}
\label{Lemma1}
If $B_0$ solves (\ref{13}), then $A_0(t)$ solves (\ref{12}) in the sense of
tempered distributions for every $t \in [0,T]$ . The solution of (\ref{12}) is
unique in $\dot{H}^{\half} \cap \dot{H}^1$ .
\end{lemma}
The regularity of $A_0$ and $B_0$ and estimates for $A_0(t)$ and $B_0(t)$ are
studied in Lemma \ref{Lemma6} , Cor. \ref{Cor.1} and Lemma \ref{Lemma7}.

We start by estimating the null forms.
The proof of the following bilinear estimates relies on estimates given by
Foschi and Klainerman \cite{FK}. We first treat the case $r>1$ , but close to
$1$.
\begin{lemma}
\label{Lemma5.1}
Let $1 < r \le 2$ .
Assume $0 \le\alpha_1,\alpha_2 $ ,  $\alpha_1+\alpha_2 \ge \frac{1}{r}$ and $ b
> \frac{1}{r}$. Let 
$$q_{12} (u,v) := Q_{12}(D^{-1}u,D^{-1}v) \, , $$
where the null form $Q_{12}$ is given by (\ref{q}).
The following estimate applies
$$ \|q_{12}(u,v)\|_{H^r_{0,0}} \lesssim \|u\|_{H^r_{\alpha_1,b}}
\|v\|_{H^r_{\alpha_2,b}} \, . $$
\end{lemma}
\begin{proof}
	Because we use inhomogeneous norms it is obviously possible to assume $\alpha_1
+ \alpha_2 = \frac{1}{r}$ . Moreover, by interpolation we may reduce to the case
$\alpha_1= \frac{1}{r}$ , $\alpha_2 =0$ .
	
The left hand side of the claimed estimate equals
$$ \|{\mathcal F}(q_{12}(u,v))\|_{L^{r'}_{\tau \xi}} = \| \int
q_{12}(\eta,\eta-\xi) \tilde{u}(\lambda,\eta) \tilde{v}(\tau - \lambda,\xi -
\eta) d\lambda d\eta \|_{L^{r'}_{\tau \xi}} \, . $$
Let now $u(t,x) = e^{\pm_1 iD} u_0^{\pm_1}(x)$ , $v(t,x) = e^{\pm_2 iD}
v_0^{\pm_2}(x)$ , so that 
$$ \tilde{u}(\tau,\xi) = c \delta(\tau \mp_1 |\xi|) \widehat{u_0^{\pm_1}}(\xi)
\quad , \quad \tilde{v}(\tau,\xi) = c \delta(\tau \mp_2 |\xi|)
\widehat{v_0^{\pm_2}}(\xi) \, . $$
This implies
\begin{align*}
&\|{\mathcal F}(q_{12}(u,v))\|_{L^{r'}_{\tau \xi}} \\
&= c^2 \| \int q_{12}(\eta,\eta-\xi) \widehat{u_0^{\pm_1}}(\eta)
\widehat{v_0^{\pm_2}}(\xi-\eta) \,\delta(\lambda \mp_1 |\eta|)
\delta(\tau-\lambda\mp_2|\xi-\eta|) d\lambda d\eta \|_{L^{r'}_{\tau \xi}} \\
& = c^2 \| \int q_{12}(\eta,\eta-\xi) \widehat{u_0^{\pm_1}}(\eta)
\widehat{v_0^{\pm_2}}(\xi-\eta) \,\delta(\tau\mp_1|\eta| \mp_2|\xi-\eta|) d\eta
\|_{L^{r'}_{\tau \xi}} \, .
\end{align*}
By symmetry we only have to consider the  elliptic case $\pm_1=\pm_2 = +$ and
the hyperbolic case $\pm_1= + \, , \, \pm_2=-$ .  \\
{\bf Elliptic case.} We obtain by \cite{FK}, Lemma 13.2:
$$|q_{12}(\eta,\xi-\eta)| \le \frac{|\eta_1 (\xi - \eta)_2 - \eta_2
(\xi-\eta)_1|}{|\eta| \, |\xi - \eta|} \lesssim \frac{|\xi|^{\half} (|\eta| +
|\xi - \eta| - |\xi|)^{\half}}{|\eta|^{\half} |\xi - \eta|^{\half}} \, . $$
By H\"older's inequality we obtain
\begin{align*}
&\|{\mathcal F}(q_{12}(u,v))\|_{L^{r'}_{\tau \xi}} \\
& \lesssim \|\int \frac{|\xi|^{\half} ||\tau|-|\xi||^{\half}}{|\eta|^{\half}
|\xi - \eta|^{\half}} \,
 \delta(\tau-|\eta|-|\xi - \eta|) \, |\widehat{u_0^+}(\eta)| \,
|\widehat{v_0^+}(\xi - \eta)| d\eta \|_{L^{r'}_{\tau \xi}} \\
& \lesssim \sup_{\tau,\xi} I  \,\, \|\widehat{D^{\frac{1}{r}} u_0^+}\|_{L^{r'}}
\| \widehat{v_0^+}\|_{L^{r'}} \, ,
\end{align*}
where
$$ I = |\xi|^{\half} ||\tau|-|\xi||^{\half} \left( \int \delta(\tau - |\eta| -
|\xi - \eta|) \, |\eta|^{-1-\frac{r}{2}} |\xi - \eta|^{-\frac{r}{2}} d\eta
\right)^{\frac{1}{r}} \, . $$
We want to prove $ \sup_{\tau,\xi} I \lesssim 1 $ . By \cite{FK}, Lemma 4.3 we
obtain
$$\int \delta(\tau - |\eta| - |\xi - \eta|) \, |\eta|^{-1-\frac{r}{2}} |\xi -
\eta|^{-\frac{r}{2}} d\eta \sim \tau^A ||\tau|-|\xi||^B \, , $$
where $A= \max(1+\frac{r}{2},\frac{3}{2}) - 1-r= -\frac{r}{2}$ and
$B=1-\max(1+\frac{r}{2},\frac{3}{2})=-\frac{r}{2}$ . Using $|\xi| \le |\tau|$ 
this implies
$$
I  \lesssim |\xi|^{\half} ||\tau|-|\xi||^{\half} \tau^{-\half}
||\tau|-|\xi||^{-\half}\le  1 \, .$$
{\bf Hyperbolic case.} We start with the following bound (cf. \cite{FK}, Lemma
13.2):
$$  |q_{12}(\eta,\xi-\eta)| \le \frac{|\eta_1 (\xi - \eta)_2 - \eta_2
(\xi-\eta)_1|}{|\eta| \, |\xi - \eta|}  \lesssim \frac{|\xi|^{\half}
(|\xi|-||\eta|-|\eta-\xi||)^{\half}}{|\eta|^{\half} |\xi-\eta|^{\half}} \, , $$
so that similarly as in the elliptic case we have to estimate
$$ I = |\xi|^{\half} ||\tau|-|\xi||^{\half} \left( \int \delta(\tau - |\eta| +
|\xi - \eta|) \, |\eta|^{-1-\frac{r}{2}} |\xi - \eta|^{-\frac{r}{2}} d\eta
\right)^{\frac{1}{r}} \, . $$
In the subcase $|\eta|+|\xi-\eta| \le 2|\xi|$ we apply \cite{FK}, Prop. 4.5 and
obtain
$$\int \delta(\tau - |\eta| + |\xi - \eta|) \, |\eta|^{-1-\frac{r}{2}} |\xi -
\eta|^{-\frac{r}{2}} d\eta \sim |\xi|^A ||\xi|-|\tau||^B \, . $$
where in the subcase $0 \le \tau \le |\xi|$ we obtain
$A=\max(\frac{r}{2},\frac{3}{2}) - 1-r = \half -r$ and $B= 1-
\max(\frac{r}{2},\frac{3}{2})= -\frac{1}{2}$. \\
This implies
$$I \lesssim |\xi|^{\half} ||\tau|-|\xi||^{\half} |\xi|^{\frac{1}{2r}-1}
||\tau|-|\xi||^{-\frac{1}{2r}} \lesssim  1 \, . $$
Similarly in the subcase $-|\xi| \le \tau \le 0$ we obtain
$A=\max(1+\frac{r}{2},\frac{3}{2})-1-r = - \frac{r}{2}$, $B= 1 -
\max(1+\frac{r}{2},\frac{3}{2}) = -\frac{r}{2} \, ,$ which implies
$$ I \sim |\xi|^{\half} ||\tau|-|\xi||^{\half} |\xi|^{-\half}
||\tau|-|\xi||^{-\half} = 1 \, .$$
In the subcase $|\eta| + |\xi-\eta| \ge 2|\xi|$ we obtain by \cite{FK}, Lemma
4.4:
 \begin{align*}
&\int \delta(\tau - |\eta| + |\xi - \eta|) \, |\eta|^{-1-\frac{r}{2}} |\xi -
\eta|^{-\frac{r}{2}} d\eta \\
&\sim ||\tau|-|\xi||^{-\half} ||\tau|+|\xi||^{-\half}\int_2^{\infty}
(|\xi|x+\tau)^{-\frac{r}{2}} (|\xi|x-\tau)^{1-\frac{r}{2}}(x^2-1)^{-\half} dx \\
&\sim  ||\tau|-|\xi||^{-\half} ||\tau|+|\xi||^{-\half} \int_2^{\infty}
(x+\frac{\tau}{|\xi|})^{-\frac{r}{2}} (x-\frac{\tau}{|\xi|})^{1-\frac{r}{2}}
(x^2-1)^{-\half} dx \, \cdot|\xi|^{1-r} \, .
\end{align*}
We remark that in fact the lower limit of the integral can be chosen as 2 by
inspection of the proof in \cite{FK}.
The integral converges, because $|\tau| \le |\xi|$ and $r > 1.$  This implies
the bound
$$ I \lesssim |\xi|^{\half}
||\tau|-|\xi||^{\half-\frac{1}{2r}}||\tau|+|\xi||^{-\frac{1}{2r}}
|\xi|^{\frac{1}{r}-1} \lesssim  1 \, . $$
Summarizing we obtain
$$\|q_{12}(u,v)\|_{H^r_{0,0}} \lesssim \|D^{\frac{1}{r}} u_0^{\pm_1}\|_{L^{r'}} 
\| v_0^{\pm_2}\|_{L^{r'}} \, . $$
By the transfer principle Prop. \ref{Prop.0.1} we obtain the claimed result. 
\end{proof}

An immediate consequence is the following corollary.
\begin{Cor}
\label{Lemma2}
Assume $1 \le s \le l+1$ , $l \ge \frac{1}{r} $ and $b> \frac{1}{r}$ . Then the
following estimate applies:
$$ \|Q_{jk}(\phi,D^{-1}(R_j A_k - R_k A_j))\|_{H^r_{s-1,0}} \lesssim
\|\phi\|_{H^r_{s,b}} \sum_j \|A_j\|_{H^r_{l,b}} \, . $$
\end{Cor}
\begin{proof}
After application of the fractional Leibniz rule we obtain by Lemma
\ref{Lemma5.1} the result as follows:
$$\|Q_{12}(\phi,D^{-1}v)\|_{H^r_{s-1,0}} \hspace{-0.3em} = \|q_{12}(D \phi,v)\|_{H^r_{s-1,0}} \hspace{-0.1em}
\lesssim \|D \phi\|_{H^r_{s-1,b}} \|v\|_{H^r_{l,b}} \lesssim 
\|\phi\|_{H^r_{s,b}} \|v\|_{H^r_{l,b}}  ,$$
if $l \ge \frac{1}{r}$ .
\end{proof}
\begin{lemma}
\label{Lemma3*}
Assume $1 \le l \le s+\half$ , $2s-l > \frac{3}{2r}$ and $b> \frac{1}{r}$ . Then
the following estimate pertains:
$$ \|D^{-1} Q_{jk}(Re \, \phi, Im \, \psi)\|_{H^r_{l-1,0}} \le
\|\phi\|_{H^r_{s,b}} \|\psi\|_{H^r_{s,b}}\, . $$
\end{lemma}
\begin{proof}
If we apply the elementary estimate (cf. (\cite{KS})) 
$$Q_{jk}(u,v) \precsim D(D^{\half}u D^{\half}v) \, ,$$
we reduce to
$$\|uv\|_{H^r_{l-1,0}} \lesssim \|u\|_{H^r_{s-\half,b}} \|v\|_{H^r_{s-\half,b}}
\, .$$
If $1 \le l \le s+\half$ the fractional Leibniz rule reduces this to
$$\|uv\|_{H^{0,0}} \lesssim \|u\|_{H^r_{s-l+\half,b}} \|v\|_{H^r_{s-\half,b}} \,
.$$
This follows from Lemma \ref{Lemma3'} if $2s-l >\frac{3}{2r}$ .
\end{proof}

\begin{lemma}
\label{Lemma3'}
Let $1<r \le 2$ and $\alpha_1,\alpha_2 \ge 0$ . If $\alpha_1+\alpha_2
>\frac{3}{2r}$ $b_1,b_2 > \frac{1}{2r}$ and $b_1+b_1 > \frac{3}{2r}$ . Then the
following estimate applies:
$$\|uv\|_{H^r_{0,0}} \lesssim \|u\|_{H^r_{\alpha_1,b_1}}
\|v\|_{H^r_{\alpha_2,b_2}} \, . $$
\end{lemma}
\begin{proof}
This follows from \cite{GT}, Prop. 3.1 by summation over the dyadic pieces.
\end{proof}

\begin{lemma}
\label{Lemma5.6}
Let $1 < r \le 2$ , $0 \le \alpha_1,\alpha_2$ and  $\alpha_1+\alpha_2 \ge
\frac{1}{r}+b$ , $b >\frac{1}{r}$ . Then the following estimate applies:
$$ \|uv\|_{H^r_{0,b}} \lesssim
 \|u\|_{H^r_{\alpha_1,b}} \|v\|_{H^r_{\alpha_2,b}} \, . $$
\end{lemma}
\begin{proof}
	We may assume $\alpha_1 = \frac{1}{r}+b$ , $\alpha_2 =0$ .
We apply the "hyperbolic Leibniz rule" (cf. \cite{AFS}, p. 128):
\begin{equation}
\label{HLR}
 ||\tau|-|\xi|| \lesssim ||\rho|-|\eta|| + ||\tau - \rho| - |\xi-\eta|| +
b_{\pm}(\xi,\eta) \, , 
\end{equation}
where
 $$ b_+(\xi,\eta) = |\eta| + |\xi-\eta| - |\xi| \quad , \quad b_-(\xi,\eta) =
|\xi| - ||\eta|-|\xi-\eta|| \, . $$

Let us first consider the term $b_{\pm}(\xi,\eta)$ in (\ref{HLR}). Decomposing 
$uv=u_+v_++u_+v_-+u_-v_++u_-v_-$ , where $u_{\pm}(t)= e^{\pm itD} f , v_\pm(t) =
e^{\pm itD} g$  , we use
$$ \widehat{u}_{\pm}(\tau,\xi) = c \delta(\tau \mp |\xi|) \widehat{f}(\xi) \quad
, \quad  \widehat{v}_{\pm}(\tau,\xi) = c \delta(\tau \mp |\xi|)
\widehat{g}(\xi)$$
and have to estimate
\begin{align*}
&\| \int  b^{b}_{\pm}(\xi,\eta) \delta(\tau - |\eta| \mp |\xi-\eta|)
\widehat{f}(\xi) \widehat{g}(\xi-\eta)d\eta \|_{L^{r'}_{\tau \xi}} \\
& = \| \int ||\tau|-|\xi||^b \delta(\tau - |\eta| \mp |\xi-\eta|)
\widehat{f}(\xi) \widehat{g}(\xi-\eta) d\eta \|_{L^{r'}_{\tau \xi}} \\
& \lesssim \sup_{\tau,\xi} I \, \|\widehat{D^{\frac{1}{r}+b} f}\|_{[L^{r'}}
\|\widehat{g}\|_{[L^{r'}} \, .
\end{align*}
Here we used H\"older's inequality, where
$$I = ||\tau|-|\xi||^b  (\int \delta(\tau-|\eta|\mp|\xi-\eta|) |\eta|^{-1-b r}
d\eta)^{\frac{1}{r}} \, . $$
In order to obtain $I \lesssim 1$ we first consider the elliptic case
$\pm_1=\pm_2=+$ and use \cite{FK}, Prop. 4.3. Thus 
$$ I \sim ||\tau|-|\xi||^b  \tau^{\frac{A}{r}} ||\tau|-|\xi||^{\frac{B}{r}} =
||\tau|-|\xi||^b ||\tau|-|\xi||^{-b} = 1$$
with $A=\max(1+br,\frac{3}{2})-(1+br)=0$ and $B=1-\max(1+br,\frac{3}{2})= -br$ .

Next we consider the hyperbolic case $\pm_1 = + \, , \, \pm_2=-$ . \\
First we assume $|\eta|+|\xi-\eta| \le 2 |\xi|$ and use \cite{FK}, Prop. 4.5
which gives
$$\int \delta(\tau - |\eta| + |\xi-\eta|) |\eta|^{-1-br}  d\eta \sim |\xi|^A
||\xi|-|\tau||^B \, , $$
where $A=\frac{3}{2}-(1+br)= \half-br$ , $B=1-\frac{3}{2}=- \half$ , if $0 \le
\tau\le |\xi|$ ,so that
$$I \sim ||\tau|-|\xi||^b |\xi|^{\frac{1}{2r}-b} ||\tau|-|\xi||^{-\frac{1}{2r}}
\lesssim 1 \, .$$   
If $-|\xi| \le \tau \le 0$ we obtain $A=\max(1+br,\frac{3}{2})-(1+br)=0$ ,
$B=1-\max(1+br,2)=-br$ , which implies  $I \lesssim 1$ .\\
Next we assume $|\eta|+|\xi-\eta| \ge 2|\xi|$ , use \cite{FK}, Lemma 4.4 and
obtain
\begin{align*}
&I \sim ||\tau|-|\xi||^b (\int \delta(\tau-|\eta|-|\xi-\eta|) |\eta|^{-1-br}
d\eta)^{\frac{1}{r}} \\
& \sim  ||\tau|-|\xi||^b \big(||\tau|-|\xi||^{-\half} ||\tau|+|\xi||^{-\half}
\int_2^{\infty} (|\xi|x + \tau)^{-br} (|\xi|x-\tau)(x^2-1)^{-\half}
dx\big)^{\frac{1}{r}} \\
& \sim ||\tau|-|\xi||^b \big(||\tau|-|\xi||^{-\half}
||\tau|+|\xi||^{-\half} \, \cdot\\
& \hspace{12em}
\cdot \int_2^{\infty} (x+\frac{\tau}{|\xi|})^{-br}
(x-\frac{\tau}{|\xi|}) (x^2-1)^{-\half} dx \,\cdot |\xi|^{1-br}\big)^{\frac{1}{r}} 
\, .
\end{align*}
This integral converges, because $\tau \le |\xi|$ and $b >\frac{1}{r}$ .This
implies
$$I \lesssim  ||\tau|-|\xi||^{b-\frac{1}{2r}} ||\tau|+|\xi||^{-\frac{1}{2r}}
|\xi|^{\frac{1}{r}-b} \lesssim  1 \, , $$
using $|\tau| \le |\xi|$ .

By the transfer principle we obtain
$$\| B_{\pm}^b (u,v)\|_{H^r_{0,0}} \lesssim \|u\|_{H^r_{\frac{1}{r}+b,b}}
\|v\|_{H^r_{0,b}} \, .$$ 
 Here $B^b_{\pm}$ denotes the operator with Fourier symbol $b_{\pm}$ . \\
Consider now the term $||\rho|-|\eta||$ (or similarly
$||\tau-\rho|-|\xi-\eta||$) in (\ref{HLR}). We have to prove
$$ \|u D_-^b v\|_{H^r_{0,0}} \lesssim \|u\|_{
H^r_{\alpha_1,b}}
\|v\|_{H^r_{\alpha_2,b}} \, , $$
which is implied by
$$\|uv\|_{H^r_{0,0}} \lesssim \|u\|_{H^r_{\alpha_1,b}}
\|v\|_{H^r_{\alpha_2,0}} \, . $$
This results from Lemma \ref{Lemma3'}, because $\alpha_1+\alpha_2\ge \frac{1}{r}
+b > \frac{2}{r}$ , which completes the proof. 
\end{proof}

\begin{lemma}
\label{Lemma3'c}
Let $1 < r \le 2$ . Assume $\alpha_1,\alpha_2 \ge 0$ and $\alpha_1+\alpha_2 >
\frac{7}{4r}$ , $b >\frac{1}{r}$ . Then the following estimate applies:
$$ \|uv\|_{H^r_{0,\frac{1}{2r}+}} \lesssim \|u\|_{H^r_{\alpha_1,b}}
\|v\|_{H^r_{\alpha_2,b}} \, . $$
\end{lemma}
\begin{proof}
Interpolation between the estimates in  Lemma \ref{Lemma3'} and Lemma
\ref{Lemma5.6} implies the result.
\end{proof}

Let us now consider the cubic nonlinearities.
\begin{lemma}
\label{Lemma4}
Let $1 < r \le 2$ . If $1 \le s \le l+1$ , $l > \frac{13}{8r} - \half$ , $2l-s >
\frac{7}{4r}-1 $ and $b > \frac{1}{r}$ , the following estimate applies:
$$ \|A^j A_j \phi \|_{H^r_{s-1,0}} \lesssim \|A^j\|_{H^r_{l,b}}
\|A_j\|_{H^r_{l,b}} \|\phi\|_{H^r_{s,b}}\, . $$
\end{lemma}
\begin{proof}
We obtain
$$\|uvw\|_{H^r_{s-1,0}} \lesssim \|u\|_{H^r_{s,b}} \|vw\|_{H^r_{m,\frac{1}{2r}}}
\lesssim \|u\|_{H^r_{s,b}} \|v\|_{H^r_{l,b}} \|w\|_{H^r_{l,b}}
\, . $$
For the first estimate we need $m+1 \ge s \ge  1$ and $1-s+s-m > \frac{3}{2r} $
, if we use Lemma \ref{Lemma3'} , thus $ m > \frac{3}{2r}-1$ . For the second
estimate we apply Lemma \ref{Lemma3'c}, which requires $ l \ge m$ and $2l-m >
\frac{7}{4r}$ . This implies the conditions $2l-(s-1)  > \frac{7}{4r}$
$\Leftrightarrow$ $2l-s > \frac{7}{4r}-1$ and $l > \frac{7}{8r} +
\frac{3}{4r}-\half = \frac{13}{8r}-\half$ .
\end{proof}

\begin{lemma}
\label{Lemma5}
Let $1<r\le 2$ . Assume $1 \le l \le s+1$ , $s >  \frac{13}{8r} - \half$ , $2s-l
> \frac{7}{4r} -1$ and $b >\frac{1}{r}$ . Then
$$ \| |\phi|^2 A_j \|_{H^r_{l-1,0}} \lesssim \|\phi\|_{H^r_{s,b}} 
\|\phi\|_{H^r_{s,b}} \|A_j\|_{H^r_{l,b}} \, . $$
\end{lemma}
\begin{proof}
We obtain
$$\|uvw\|_{H^r_{l-1,0}} \lesssim \|u\|_{H^r_{l,b}} \|vw\|_{H^r_{m,\frac{1}{2r}}}
\lesssim \|u\|_{H^r_{l,b}} \|v\|_{H^r_{s,b}} \|w\|_{H^r_{s,b}}
\, . $$
For the first estimate we need $m+1 \ge l \ge  1$ and $ m > \frac{3}{2r}-1$, if 
we use Lemma \ref{Lemma3'} . For the second estimate we apply Lemma
\ref{Lemma3'c} , which requires $ s \ge m$ and $2s-m > \frac{7}{4r}$ . This
implies the conditions $2s-(l-1)  > \frac{7}{4r}$ $\Leftrightarrow$ $2s-l >
\frac{7}{4r}-1$ and $s > \frac{7}{4r} + \frac{3}{4r}-\half$ .
\end{proof}

We come now to the estimates for the elliptic part  and start with $B_0$ .
\begin{lemma}
\label{Lemma6}
1. Assume $ s > 1 $ , $\phi \in C^0([0,T],H^s)$ , $A_j \in C^0([0,T],H^{s-1})$ .
Assume $B_0$ solves
 $$\Delta B_0 = - \partial^j Im(\phi \overline{\partial_j\phi}) +
\partial^j(|\phi|^2 A_j) \, . $$
Then $B_0  \in C^0([0,T],\dot{H}^{\sigma})$ for $0 < \sigma \le s$ and
$$
\|B_0\|_{L^{\infty}((0,T),\dot{H}^{\sigma})} \lesssim
\|\phi\|^2_{L^{\infty}((0,T),H^s)} (1+\sum_j\|A_j\|_{L^{\infty}((0,T),H^{s-1})})
\, .
$$
2. Assume $\half < s \le 1$ , $A_j \in C^0([0,T],L^2)$ . Then $B_0 \in
C^0([0,T],\dot{H}^{\sigma})$ for $0 < \sigma < 2s-1$ and
$$
\|B_0\|_{L^{\infty}((0,T),\dot{H}^{\sigma})} \lesssim
\|\phi\|^2_{L^{\infty}((0,T),H^s)} (1+\sum_j\|A_j\|_{L^{\infty}((0,T),L^2)}) \,
.
$$
\end{lemma}
\begin{proof}
1. We have to prove for $0< \sigma \le s$ :
\begin{equation}
\label{6''}
 \|\phi \overline{\partial_j \phi}\|_{\dot{H}^{\sigma-1}} \lesssim
\|\phi\|_{H^s}^2 \, . 
\end{equation}
For $s > 1$ the Sobolev multiplication law (SML) (\cite{T}, Proposition 3.15)
implies
$$ \|\phi \overline{\partial_j \phi}\|_{H^{s-1}} \lesssim \|\phi\|_{H^s}
\|\partial_j \phi\|_{H^{s-1}} \lesssim \|\phi\|_{H^s}^2 $$
and by the homogeneous version of SML we obtain for $s \ge \half$ :
$$ \|\phi \overline{\partial_j \phi}\|_{\dot{H}^{-1+\epsilon}} \lesssim
\|\phi\|_{\dot{H}^{\half - \frac{\epsilon}{2}}} \|\partial_j
\phi\|_{\dot{H}^{-\half-\frac{\epsilon}{2}}} \lesssim \|\phi\|_{H^s}^2 \, .$$
This implies (\ref{6''}).
Next we have to prove for $0< \sigma \le s$
\begin{equation}
\label{6'''}
\| |\phi|^2 A_j\|_{\dot{H}^{\sigma-1}} \lesssim \|\phi\|^2_{H^s}
\|A_j\|_{H^{s-1}} \, . 
\end{equation}
The SML implies for $s > 1$ :
$$ \| |\phi|^2 A_j\|_{\dot{H}^{s-1}} \lesssim \| |\phi|^2\|_{H^{s} }
\|A_j\|_{H^{s-1}} \lesssim \|\phi\|_{H^s}^2 \|A_j\|_{H^{s-1}} $$
and by the homogeneous SML :
$$ \| |\phi|^2 A_j\|_{\dot{H}^{-1+\epsilon}} \lesssim \|
|\phi|^2\|_{\dot{H}^{\frac{\epsilon}{2}} }
\|A_j\|_{\dot{H}^{\frac{\epsilon}{2}}} \lesssim \|\phi\|_{H^{\half+\epsilon}}^2
\|A_j\|_{\dot{H}^{\frac{\epsilon}{2}}}   \lesssim \|\phi\|_{H^s}^2
\|A_j\|_{H^{s-1}} \, ,$$
which implies (\ref{6'''}).\\
2. Next we prove (\ref{6''}) for $\half < s \le 1$ and $0 < \sigma < 2s-1$ . If
the frequencies of $\partial_j \phi$ are large, we obtain by the homogeneous SML

$$ \|\phi \overline{\partial_j \phi}\|_{\dot{H}^{-1+}} \lesssim
\|\phi\|_{\dot{H}^{1-s+}} \|\partial_j \phi\|_{\dot{H}^{s-1}} \lesssim
\|\phi\|_{H^s} \|\partial_j \phi\|_{H^{s-1}} $$ 
and 
$$ \|\phi \overline{\partial_j \phi}\|_{\dot{H}^{2s-2-}} \lesssim
\|\phi\|_{\dot{H}^{s-}} \|\partial_j \phi\|_{\dot{H}^{s-1}} \lesssim
\|\phi\|_{H^s} \|\partial_j \phi\|_{H^{s-1}} \, .$$ 
If the frequencies of $\partial_j \phi$ are small we obtain
$$ \|\phi \overline{\partial_j \phi}\|_{\dot{H}^{-1+}} \lesssim
\|\phi\|_{H^{0+}} \|\partial_j \phi\|_{L^2} \lesssim \|\phi\|_{H^s} \|\partial_j
\phi\|_{H^{s-1}} $$ 
and 
$$ \|\phi \overline{\partial_j \phi}\|_{\dot{H}^{2s-2-}} \lesssim
\|\phi\|_{\dot{H}^{2s-1-}} \|\partial_j \phi\|_{L^2} \lesssim
\|\phi\|_{H^{2s-1}} \|\partial_j \phi\|_{H^{s-1}}
\lesssim \|\phi\|_{H^s} \|\partial_j \phi\|_{H^{s-1}} \, .$$ 
Next we prove (\ref{6'''}). We obtain
$$ \| |\phi|^2 A_j\|_{\dot{H}^{2s-2-}} \lesssim \| |\phi|^2\|_{\dot{H}^{2s-1-} }
\|A_j\|_{L^2} \lesssim \|\phi\|_{H^{s-}}^2 \|A_j\|_{L^2}   \lesssim
\|\phi\|_{H^s}^2 \|A_j\|_{L^2} \, .$$
and
$$ \| |\phi|^2 A_j\|_{\dot{H}^{-1+}} \lesssim \| |\phi|^2\|_{\dot{H}^{0+} }
\|A_j\|_{L^2} \lesssim \|\phi\|_{H^{\half +}}^2 \|A_j\|_{L^2}   \lesssim
\|\phi\|_{H^s}^2 \|A_j\|_{L^2} \, .$$
\end{proof}
\begin{Cor}
\label{Cor.1}
Let $s$ and $\sigma$ be as in Lemma \ref{Lemma6}. Then $A_0 \in
C^0([0,T],\dot{H}^{\sigma})$ for $\sigma \le 1$ and
$$ \|A_0\|_{L^{\infty}((0,T),\dot{H}^{\sigma})} \lesssim T
\|\phi\|^2_{L^{\infty}((0,T),H^s)} ( 1 + \sum_j
\|A_j\|_{L^{\infty}((0,T),H^{\max(s-1,0)})} ) + \|a_0\|_{\dot{H}^{\sigma}} 
\,.$$
\end{Cor}
\begin{proof}
We only have to recall $A_0(t) := \int_0^t B_0(s) ds + a_0$ ,  and $a_0$ as the
solution of the variational problem considered before Lemma \ref{Lemma1} belongs
to $H^1$ .
\end{proof}
It is possible to improve the regularity of $A_0$ by using that $A_0$ solves the
elliptic equation (\ref{12}) by Lemma \ref{Lemma1}.
\begin{lemma}
\label{Lemma7}
1. If $s > 1$ and $\phi \in C^0([0,T],H^s) \cap C^1([0,T],H^{s-1})$ then $A_0
\in C^0((0,T),\dot{H}^{s+1})$ and the following estimate applies:
\begin{align*}
\|A_0\|_{L^{\infty}((0,T),\dot{H}^{s+1})}  &\lesssim 
\|\phi\|_{L^{\infty}((0,T),H^s)} \|\partial_t \phi\|_{L^{\infty}((0,T),H^{s-1})}
\\
& \quad + \|\phi\|_{L^{\infty}((0,T),H^s)}^2
\|A_0\|_{L^{\infty}((0,T),\dot{H}^{0+})} \, .
\end{align*}
2. If $\half < s \le 1$ and $A_0 \in C^0([0,T],\dot{H}^{\sigma})$ for some $0 <
\sigma < 2s-1$ , then
$A_0 \in C^0((0,T),\dot{H}^a)$ for $1 < a < 2s$ , and the following estimate
applies:
\begin{align*}
\|A_0\|_{L^{\infty}((0,T),\dot{H}^a)}  &\lesssim 
\|\phi\|_{L^{\infty}((0,T),H^s)} \|\partial_t \phi\|_{L^{\infty}((0,T),H^{s-1})}
\\
& \quad + \|\phi\|_{L^{\infty}((0,T),H^s)}^2
\|A_0\|_{L^{\infty}((0,T),\dot{H}^{\sigma})} \, .
\end{align*}
\end{lemma}
\begin{proof}
1. By the elliptic equation (\ref{12}) we obtain
$$ \|A_0\|_{\dot{H}^{s+1}} \lesssim \|\phi \partial_t \phi\|_{\dot{H}^{s-1}} +
\| |\phi|^2 A_0\|_{\dot{H}^{s-1}} \, .$$
By the SML we obtain
$$  \|\phi \partial_t \phi\|_{\dot{H}^{s-1}} \lesssim \|\phi\|_{H^s}
\|\partial_t \phi\|_{H^{s-1}}$$ 
and 
$$  \| |\phi|^2 A_0\|_{\dot{H}^{s-1}} \lesssim \|A_0\|_{\dot{H}^{0+}} \|
|\phi|^2\|_{\dot{H}^{s-}} \lesssim \|A_0\|_{\dot{H}^{0+}}  \|\phi\|_{H^s}^2 \, .
$$
2. This can be proven similarly as for Lemma \ref{Lemma6}.
\end{proof}

Next we address the bilinear terms involving $A_0$ and $\partial_t A_0$ .

\begin{lemma}
\label{Lemma8}Let $1 < r \le 2$ , $s>1$ and $b> \frac{1}{r}$ .\\
1. If $s \le 2$ the following estimate applies:
$$\|A_0 \partial_t \phi\|_{H^r_{s-1,0}} \lesssim
\|A_0\|_{L^{\infty}((0,T),\dot{H}^{1+}\cap \dot{H}^{0+})} \|\partial_t
\phi\|_{H^r_{s-1,b}} \, . $$
2. If $s > 2$ the following inequality applies:
\begin{equation}
\label{21}
\|A_0 \partial_t \phi\|_{H^r_{s-1,0}} \lesssim
\|A_0\|_{L^{\infty}((0,T),\dot{H}^{s-1}\cap \dot{H}^{0+})} \|\partial_t
\phi\|_{H^r_{s-1,b}} \, .
\end{equation}
\end{lemma}
\begin{proof}
1. We want to use the fractional Leibniz rule and estimate by H\"older and Young
:
\begin{align*}
\|A_0 \Lambda^{s-1} \partial_t \phi\|_{H^r_{0,0}} &= \|\widehat{A_0
\Lambda^{s-1} \partial_t \phi}\|_{L^{r'}_{\tau \xi}} \\
&\lesssim \|\widehat{A_0}\|_{L^1_{\xi} L^2_{\tau}} \| \|\langle|\tau|-|\xi|
\rangle^{-b}\|_{L^2_\tau} \| \langle |\tau|-|\xi| \rangle^b
\widehat{\Lambda^{s-1} \partial_t \phi}\|_{L^{r'}_{\tau}} \|_{L^{r'}_\xi} \\
&\lesssim \|\widehat{A_0}\|_{L^1_{\xi}L^2_{\tau}} \|\partial_t
\phi\|_{H^r_{s-1,b}} \, .
\end{align*}
Now for high frequencies of $A_0$ we obtain
$$\|\widehat{A_0}\|_{L^1_{\xi}L^2_{\tau}} \lesssim \| \langle \xi
\rangle^{-1-}\|_{L^2_\xi} \| \langle \xi \rangle^{1+} \widehat{A_0}\|_{L^2_\tau
L^2_\xi} \lesssim \|A_0\|_{L^2(0,T),\dot{H}^{1+})} \lesssim
\|A_0\|_{L^{\infty}((0,T),\dot{H}^{1+})}$$
and for low frequencies of $A_0$ :
$$ \|\widehat{A_0}\|_{L^1_{\xi}L^2_{\tau}} \lesssim \| |\xi |^{0-}\|_{L^2_{|\xi|
\le 1}} \| | \xi |^{0+} \widehat{A_0}\|_{L^2_\xi L^2_\tau}\lesssim
\|A_0\|_{L^{\infty}((0,T),\dot{H}^{0+})} \, .$$
Thus we obtain
$$\|A_0 \Lambda^{s-1} \partial_t \phi\|_{H^r_{0,0}} \lesssim
\|A_0\|_{L^{\infty}((0,T),\dot{H}^{1+}\cap \dot{H}^{0+})} \|\partial_t
\phi\|_{H^r_{s-1,b}} \, . $$
Next we estimate $\|\Lambda^{s-1} A_0  \partial_t \phi\|_{H^r_{0,0}}$ . For low
frequencies we obtain as before 
$$ \|\Lambda^{s-1} A_0  \partial_t \phi\|_{H^r_{0,0}} \lesssim
\|A_0\|_{L^{\infty}((0,T),\dot{H}^{1+} \cap \dot{H}^{0+})} \|\partial_t
\phi\|_{H^r_{s-1,b}} $$
and for high frequencies of $A_0$ we obtain by H\"older and Young:
\begin{align*}
&\| \widehat{\Lambda^{s-1} A_0 \partial_t \phi}\|_{L^{r'}_{\tau \xi}} \\
& \lesssim \| \langle \xi \rangle^{s-2-}\|_{L_{\xi}^{\frac{2}{2-s}-}} \| \langle
\xi \rangle^{2-s+} \widehat{A_0}\|_{L^2_\xi L^2_{\tau}} \|\langle \xi
\rangle^{-(s-1)} \langle |\tau|-|\xi| \rangle^{-b} \|_{L_{\xi}^{\frac{2}{s-1}+}
L^2_{\tau}}  \times  \\
&\qquad \times\| \langle \xi \rangle^{s-1} \langle |\tau|-|\xi| \rangle^{b}
\widehat{\partial_t \phi}\|_{L^{r'}_{\tau \xi}} \\
&\lesssim \|A_0\|_{L^{\infty}((0,T),\dot{H}^{1+})} \|\widehat{\partial_t
\phi}\|_{H^r_{s-1,b}} \, , 
\end{align*}
because $1+\frac{1}{r'} = \frac{2-s}{2}+\half + \frac{s-1}{2} + \frac{1}{r'}$ .
By the fractional  Leibniz rule we obtain the claimed estimate. \\
2. If  $s>2$  we handle the integration with respect to $\tau$ as in 1. We
obtain for high frequencies of $A_0$ :
\begin{align*}
\| \widehat{\Lambda^{s-1} A_0 \partial_t \phi}\|_{L^{r'}_{\xi}} & \lesssim
\|\widehat{\Lambda^{s-1} A_0}\|_{L^2_{\xi}} \| \langle \xi
\rangle^{-(s-1)}\|_{L^2_{\xi}} \| \langle \xi \rangle^{s-1} \widehat{\partial_t
\phi}\|_{L^{r'}_{\xi}} \\
& \lesssim \|A_0\|_{\dot{H}^{s-1}} \|\widehat{\Lambda^{s-1}\partial_t
\phi}\|_{L^{r'}_{\xi}} 
\end{align*}
and
\begin{align*}
\| A_0 \partial_t \Lambda^{s-1} \phi\|_{{L^{r'}_{\xi}}} &  \lesssim  \| \langle
\xi \rangle^{-(s-1)}\|_{L^2_{\xi}} \|\langle \xi
\rangle^{s-1}\widehat{A_0}\|_{L^2_{\xi}} \|  \widehat{\Lambda^{s-1}{\partial_t
\phi}}\|_{L^{r'}_{\tau \xi}} \\
& \lesssim \|A_0\|_{\dot{H}^{s-1}} \|\widehat{\Lambda^{s-1}\partial_t
\phi}\|_{{L^{r'}_{\xi}}} \, ,
\end{align*}
so that the fractional Leibniz rule implies (\ref{21}). Low frequencies of $A_0$
are handled as in part 1. 
\end{proof}

\begin{lemma}
\label{Lemma9}
Let $1 < r \le 2 $ and $s > 1$ . Then
\begin{equation}
\label{22}
\|(\partial_t A_0) \phi\|_{H^r_{s-1,0}} \lesssim \|\partial_t
A_0\|_{L^{\infty}((0,T),\dot{H}^{0+} \cap\dot{H}^{s-1})} \|\phi\|_{H^r_{s,b}}\,
.
\end{equation} 
\end{lemma}
\begin{proof} We handle the $\tau$-integration as in the previous lemma.
We apply the fractional Leibniz rule and estimate first for high frequencies of
$\partial_t A_0$ :
\begin{align*}
\|\widehat{(\Lambda^{s-1}\partial_t A_0) \phi} \|_{L^{r'}_{\xi}}  
\lesssim \|\widehat{\Lambda^{s-1} \partial_t A_0}\|_{L^2_{\xi}} \| \langle \xi
\rangle^{-s}\|_{L^2_{\xi}} \| \langle \xi \rangle^s \widehat{\phi}
\|_{L^{r'}_{\xi}} 
\lesssim \|\partial_t A_0\|_{\dot{H}^{s-1}} \|\widehat{\Lambda^s \phi}
\|_{L^{r'}_{\xi}} 
\end{align*}
and for low frequencies of $\partial_t A_0$ :
\begin{align*}
\| \widehat{(\Lambda^{s-1} \partial_t A_0) \phi}\|_{L^{r'}_{\xi}} &\lesssim \|
|\xi|^{0-}\|_{L^2_{|\xi| \le 1}} \| |\xi|^{0+} \widehat{\partial_t
A_0}\|_{L^2_{\xi}} \|\phi\|_{L^{r'}_{\xi}} \\
& \lesssim \| \partial_t A_0 \|_{\dot{H}^{0+}} \|\phi\|_{L^{r'}_{\xi}} \, .
\end{align*}
Moreover we obtain for high frequencies of $\partial_t A_0$ :
\begin{align*}
\|\widehat{(\partial_t A_0) \Lambda^{s-1} \phi}\|_{L^{r'}_{\xi}} &
\lesssim \| \langle \xi \rangle^{0-}\|_{L^{\infty-}_{\xi}} \| \langle \xi
\rangle^{0+} \widehat{\partial_t A_0}\|_{L^2_{\xi}} \| \langle \xi
\rangle^{-1}\|_{L^{2+}_{\xi}} \| \langle \xi \rangle \Lambda^{s-1}
\phi\|_{L^{r'}_{\xi}} \\
& \lesssim\| \partial_t A_0\|_{\dot{H}^{0+}} \|\widehat{\Lambda^s \phi}
\|_{L^{r'}_{\xi}}\, ,	
\end{align*}
whereas for low frequencies of $\partial_t A_0$ we obtain
\begin{align*}
\| \widehat{( \partial_t A_0) \Lambda^{s-1} \phi}\|_{L^{r'}_{\xi}} &\lesssim \|
|\xi|^{0-}\|_{L^2_{|\xi| \le 1}} \| |\xi|^{0+} \widehat{\partial_t
A_0}\|_{L^2_{\xi}} \|\Lambda^{s-1} \phi\|_{L^{r'}_{\xi}} \\
& \lesssim \| \partial_t A_0 \|_{\dot{H}^{0+}}  \|\Lambda^{s-1}
\phi\|_{L^{r'}_{\xi}}  \, .
\end{align*}
\end{proof}

Finally we treat the cubic term involving $A_0$ .
\begin{lemma}
\label{Lemma10}
Let $1 < r \le 2$ , $s > 1$ and $b >\frac{1}{r}$ . Then the following estimate
applies:
$$ \|A_0^2 \phi\|_{H^r_{s-1,0}} \lesssim
\|A_0\|^2_{L^{\infty}((0,T),\dot{H}^{0+} \cap \dot{H}^{s-\half+})}
\|\phi\|_{H^r_{s,b}} \, . $$
\end{lemma}
\begin{proof}
By Young's and H\"older's inequality we obtain
\begin{align*}
 \|\widehat{A_0^2 \phi}\|_{L^{r'}_{\tau\xi}} &\lesssim
\|\widehat{A_0^2}\|_{L^p_{\xi} L^2_{\tau}} \| \|\langle |\tau|-|\xi|
\rangle^{-b} \|_{L^2_{\tau}} \| \langle |\tau|-|\xi| \rangle^b     
\widehat{\phi}\|_{L^{r'}_{\tau}} \|_{L^q_{\xi}} \\
 &\lesssim \|\widehat{A_0^2}\|_{L^p_{\xi} L^2_t}\| \langle \xi
\rangle^{-1}\|_{L^{2+}_{\xi}} \| \langle \xi \rangle \langle |\tau|-|\xi|
\rangle^b      \widehat{\phi}\|_{L^{r'}_{\tau \xi}} \\
& \lesssim \| \langle \xi \rangle^{0-}\|_{L^{\infty-}_{\xi}}      \|\langle \xi
\rangle^{0+}\widehat{A_0^2}\|_{L^2_{\xi \tau}}           \|\phi\|_{H^r_{1,b}}   
    \\
& \lesssim \| \Lambda^{0+}(A_0^2)\|_{L^2_{xt}} \|\phi\|_{H^r_{1,b}} \\
& \lesssim \|A_0\|_{L^4_{xt}} \|\Lambda^{0+}A_0\|_{L^4_{xt}}
\|\phi\|_{H^r_{1,b}} \\
& \lesssim \|A_0\|_{L^{\infty}((0,T),\dot{H}^{\half})}  
\|A_0\|_{L^{\infty}((0,T),\dot{H}^{\half} \cap \dot{H}^{\half+})} 
\|\phi\|_{H^r_{1,b}}\, , 
\end{align*}
where  $\frac{1}{q}= \frac{1}{r'}+\half-$ and $\frac{1}{p}= \half +$ .
An application of the fractional Leibniz rule implies the claimed result.
\end{proof}

In the following we prove the necessary estimates in the case of data in
$L^2$-based Sobolev spaces.

The following bilinear estimates for wave-Sobolev spaces were proven in
\cite{AFS}, Lemma 7.
\begin{prop}
\label{Prop.7.1}
	Let $s_0,s_1,s_2 \in \R$ , $b_0,b_1,b_2 \ge 0$ . Assume that
	 \begin{align*}
	 b_0+b_1+b_2 &> \half \\
	 s_0+s_1+s_2 &>\frac{3}{2}-(b_0+b_1+b_2) \\
	 s_0+s_1+s_2 &> 1 - \min_{i \neq j} (b_i+b_j) \\
	s_0+s_1+s_2 &> \half - \min_i b_i  \\
	s_0+s_1+s_2& > 1-\min(b_0+s_1+s_2,s_0+b_1+s_2,s_0+s_1+b_2) \\
	s_0+s_1+s_2 & \ge \frac{3}{4} \\
	\min_{i \neq j} (s_i+s_j) & \ge 0 \, ,
	\end{align*}
where the last two inequalities are not both equalities. Then the following
estimate applies:
$$ \|uv\|_{H^{-s_0,-b_0}} \lesssim \|u\|_{H^{s_1,b_1}} \|v\|_{H^{s_2,b_2}} \,
.$$
If $b_0 <0$ , this remains true provided we  additionally assume $b_0+b_1 > 0$ ,
$b_0+b_2 > 0$ and $s_1+s_2 > -b_0$ .
\end{prop}

\begin{Cor}
	If $b_0 \ge 0$ , $b_1,b_2 > \half$ the following assumptions are sufficient:
	$$s_0+s_1+s_2 > 1-(b_0+s_1+s_2) \, , \, s_0+s_1+s_2 \ge \frac{3}{4} \,, \,
\min_{i \neq j}(s_i+s_j) \ge 0 \, ,$$  where the last two inequalities are not
both equalities.
\end{Cor}

\begin{lemma}
\label{Lemma1''}Assume $ s \le l+1$ , $ l > \frac{1}{4}$ , $s > \half$ . The
following estimate applies for $ \epsilon >0$ sufficiently small:
$$ \|Q_{12}(u,D^{-1}v)\|_{H^{s-1,-\half + 2\epsilon}} \lesssim
\|u\|_{X^{s,\half+\epsilon,\pm_1}} \|v\|_{X^{l,\frac{3}{4}+\epsilon,\pm_2}} \, . $$
\end{lemma}
\begin{proof}
By \cite{KS} we obtain the estimate
\begin{align}
\label{KS}
Q_{12}(u,D^{-1}v) & \precsim D^{\half} D_-^{\half} (D^{\half} u D^{-\half} v) +
D^{\half}(D^{\half}_- D^{\half}u D^{-\half}v) + D^{\half}(D^{\half}u D^{\half}_-
D^{-\half}v) 
\end{align}
For high frequencies of $v$ we have to prove the following estimates:
\begin{align*}
\|uv\|_{H^{s-\half,2\epsilon}} & \lesssim \|u\|_{H^{s-\half,\half+\epsilon}}
\|v\|_{H^{l+\half,\frac{3}{4}+\epsilon}} \\
\|uv\|_{H^{s-\half,-\half+2\epsilon}} & \lesssim \|u\|_{H^{s-\half,\epsilon}}
\|v\|_{H^{l+\half,\frac{3}{4}+\epsilon}} \\
\|uv\|_{H^{s-\half,-\half+2\epsilon}} & \lesssim
\|u\|_{H^{s-\half,\half+\epsilon}} \|v\|_{H^{l+\half,\frac{1}{4}+\epsilon}} 
\end{align*}
We apply Prop. \ref{Prop.7.1}. Using the notation of this proposition we require
in all these estimates $s_0+s_1+s_2 = l+\half > \frac{3}{4}$ , thus $l >
\frac{1}{4}$ . The first estimate requires $l+\half > 1-(s-\half+l+\half)$ ,
thus $2l+s > \half$ , which is fulfilled. For the second estimate we need
$l+\half > 1-(\half-s+l+\half)$ , thus $2l+s > -\half$ . For the third estimate
we need the condition $l+\half > 1-(\half-s+s-\half+\frac{1}{4})$
$\Leftrightarrow$ $l >\frac{1}{4}$ . Here is the point where it is important to
have $A_j \in H^{l,\frac{3}{4}+\epsilon}$ instead of $H^{l,\half+\epsilon}$ ,
which would lead to the condition $ l > \half$ . \\
For low frequencies of $v$ we use the elementary estimate
$$ Q_{12}(u,D^{-1}v) \precsim D^{\half}((D^{\half} u) v) \, , $$
so that we reduce to
$$\|uv\|_{H^{s-\half,-\half+2\epsilon}} \lesssim
\|u\|_{H^{s-\half,\half+\epsilon}} \|v\|_{H^{l,\half+\epsilon}} \, . $$
This however is clear, because $\|v\|_{H^{l,\half+\epsilon}} \sim
\|v\|_{H^{N,\half+\epsilon}}$ for arbitrary $N$ .
\end{proof}

\begin{lemma}
\label{Lemma3''}Assume $ s \ge l$ ,  $s > \half$ and $4s-l > \frac{7}{4}$ . The
following estimate applies for $ \epsilon >0$ sufficiently small:
$$ \|D^{-1} Q_{12}(Re \, \phi, Im \, \psi)\|_{H^{l-1,-\frac{1}{4} + 2\epsilon}}
\lesssim \|\phi\|_{X^{s,\half+\epsilon,\pm_1}} \|\psi\|_{X^{s,\half+\epsilon,\pm_2}}\, . $$
\end{lemma}
\begin{proof}
We apply (\ref{KS}). Combining this with the elementary estimate (cf. \cite{KS})
\begin{equation}
\label{KS1}
 Q_{12}(u,v) \precsim D(D^{\half} u D^{\half} v) \, , 
\end{equation}
we obtain
\begin{align*}
Q_{12}(u,v) & \precsim D^{\frac{3}{4}+2\epsilon} D_-^{\frac{1}{4}-2\epsilon}
(D^{\half} u D^{\half} v) + D^{\half}(D^{\half}_- D^{\half}u D^{\half}v)
+D^{\half}( D^{\half}u D^{\half}_- D^{\half}v) \, .
\end{align*}
For high frequencies of the product we have to prove the following estimates:
\begin{align*}
\|uv\|_{H^{l-\frac{5}{4}+2\epsilon,0}} & \lesssim
\|u\|_{H^{s-\half,\half+\epsilon}} \|v\|_{H^{s-\half,\half+2\epsilon}} \\
\|uv\|_{H^{l-\frac{3}{2},-\frac{1}{4}+2\epsilon}} & \lesssim
\|u\|_{H^{s-\half,\epsilon}} \|v\|_{H^{s-\half,\half+\epsilon}} \, . 
\end{align*}
Now apply Prop. \ref{Prop.7.1}.
We require for the first estimate $\frac{5}{4}-l+2s-1  > \frac{3}{4}$ , thus
$2s-l > \frac{1}{2}$ , which is true by our assumptions $s\ge l$ and $s > \half$
. Moreover we need $2s-l+\frac{1}{4} > 1-(2s-1) \, \Leftrightarrow \, 4s-l >
\frac{7}{4} $ as assumed. For the second estimate we need $2s-l+\half >
\frac{3}{4} \, \Leftrightarrow 2s-l > \frac{1}{4}$ , which is true, and $2s-l +
\half > 1-(\frac{3}{2}-l+s-\half)$, thus $3s-2l > -\half$ , which applies,
because $s \ge l$ and $s > \half$ . \\
For low frequencies of the product we use the elementary estimate (\ref{KS1}) ,
so that we reduce to
$$\|uv\|_{H^{l-1,-\frac{1}{4}+2\epsilon}} \lesssim
\|uv\|_{H^{-N,-\frac{1}{4}+2\epsilon}} \lesssim
\|u\|_{H^{s-\half,\half+\epsilon}} \|v\|_{H^{s-\half,\half+\epsilon}} \, . $$
This however is clear, because $N$ may be chosen arbitrarily. 
\end{proof}

\begin{lemma}
\label{Lemma8''}
The following estimate applies for $s \ge 0$ :
$$ \|A_0 \partial_t \phi\|_{H^{s-1}} \lesssim \|A_0\|_{\dot{H}^{\max(s-1,1+)}
\cap \dot{H}^{0+}} \|\partial_t\phi\|_{H^{s-1}} \, . $$
\end{lemma}
\begin{proof}
1. Let $s\le 1$ . Then by SML for $s < 1$ :
\begin{align*}
\| D^{1-s}u v\|_{L^2} &\lesssim \|D^{1-s}u\|_{\dot{H}^{s+}}
\|v\|_{\dot{H}^{1-s-}} = \|D^{1+} u\|_{L^2} \|v\|_{H^{1-s}} \,  ,\\
\|u D^{1-s}v\|_{L^2} &\lesssim \|u\|_{L^{\infty}} \|D^{1-s}v\|_{L^2} \lesssim
\|u\|_{\dot{H}^{1+} \cap \dot{H}^{0+}} \|v\|_{H^{1-s}} \, , \\
\|uv\|_{L^2} & \lesssim \|u\|_{L^{\infty}} \|v\|_{L^2}   \lesssim
\|u\|_{\dot{H}^{1+} \cap \dot{H}^{0+}} \|v\|_{L^2} \, ,
\end{align*}
where the last estimate suffices for $s=1$ ,
thus
$$ \|uv\|_{H^{1-s}} \lesssim \|u\|_{\dot{H}^{1+} \cap \dot{H}^{0+}}
\|v\|_{H^{1-s}}\, , $$
so that by duality we obtain
$$\|uv\|_{H^{s-1}} \lesssim \|u\|_{\dot{H}^{1+} \cap \dot{H}^{0+}}
\|v\|_{H^{s-1}}\, . $$
2. Let $1 < s \le 2$ . Similarly as before we obtain
\begin{align}
\nonumber
\|D^{s-1}u v \|_{L^2} & \lesssim \|D^{s-1} u\|_{\dot{H}^{2-s+}}
\|v\|_{\dot{H}^{s-1-}} \lesssim \|u\|_{\dot{H}^{1+}} \|v\|_{H^{s-1}} \, ,\\
\label{8'''}
\|u D^{s-1}v\|_{L^2}  & \lesssim \|u\|_{L^{\infty}} \|D^{s-1}v\|_{L^2} \lesssim
\|u\|_{\dot{H}^{1+} \cap \dot{H}^{0+}} \|v\|_{H^{s-1}} \,  ,
\end{align}
thus
$$
 \|uv\|_{H^{s-1}} \lesssim \|u\|_{\dot{H}^{1+} \cap \dot{H}^{0+}}
\|v\|_{H^{s-1}}\, . 
$$
3. For $s > 2$ we use 
\begin{align*}
\|D^{s-1}u v \|_{L^2} & \lesssim \|D^{s-1} u\|_{L^2} \|v\|_{L^{\infty}} \lesssim
\|D^{s-1}u\|_{L^2}  \|v\|_{H^{s-1}} \,  ,
\end{align*}
thus combining this with (\ref{8'''}) :
\begin{equation}
\label{8''}
 \|uv\|_{H^{s-1}} \lesssim \|u\|_{\dot{H}^{\max(s-1,1+)} \cap \dot{H}^{0+}}
\|v\|_{H^{s-1}}\, . 
\end{equation}
\end{proof}

\begin{lemma}
\label{Lemma9''}
The following estimate applies for $s > 0$ :
$$ \|(\partial_t A_0)  \phi\|_{H^{s-1}} \lesssim \|\partial_t
A_0\|_{\dot{H}^{\max(s-1,0+)} \cap \dot{H}^{0+}} \|\phi\|_{H^s} \, . $$
\end{lemma}
\begin{proof}
1. If $s \le 1$ we obtain by SML:
$$\|uv\|_{H^{s-1}} \lesssim \|uv\|_{\dot{H}^{s-1}} \lesssim \|u\|_{\dot{H}^{0+}}
\|v\|_{\dot{H}^{s-}} \lesssim \|u\|_{\dot{H}^{0+}} \|v\|_{H^s} \, . $$
2. If $ s > 1$ we obtain
\begin{align*}
\| D^{s-1} u v \|_{L^2} &\lesssim \|D^{s-1}u\|_{L^2} \|v\|_{L^{\infty}} \lesssim
\|u\|_{\dot{H}^{s-1}} \|v\|_{H^s} \, ,\\
\|u D^{s-1}v\|_{L^2} & \lesssim \|u\|_{\dot{H}^{0+}} \|D^{s-1}v\|_{\dot{H}^{1-}}
\lesssim  \|u\|_{\dot{H}^{0+}}  \|uv\|_{H^s} \, , \\
\|uv\|_{L^2} &\lesssim \|u\|_{\dot{H}^{0+}} \|v\|_{\dot{H}^{1-}} \lesssim 
\|u\|_{\dot{H}^{0+}} \|v\|_{H^s} \, ,
\end{align*}
thus
$$
 \|uv\|_{H^{s-1}} \lesssim \|u\|_{\dot{H}^{\max(s-1,0+)} \cap \dot{H}^{0+}}
\|v\|_{H^{s}}\, . $$
\end{proof}

In the next three lemmas we address the cubic nonlinearities.
\begin{lemma}
\label{Lemma10''}
Assume $ s > \half$ . Then the following estimate applies:
$$ \|A_0^2 \phi \|_{H^{s-1,0}} \lesssim \|A_0\|^2_{L^{\infty}((0,T),\dot{H}^{0+}
\cap \dot{H}^{\max(1,s-1)+})} \|\phi\|_{H^{s,0}}\, . $$
\end{lemma}
\begin{proof}
 In the case $ s \le 1$ we obtain
\begin{align*}
\|A_0^2 \phi \|_{H^{s-1,0}} \lesssim \|A_0^2 \phi\|_{L^2_{xt}}
 \lesssim \|A_0\|^2_{L^{\infty}_{xt}} \|\phi\|_{H^{0,0}}
\lesssim \|A_0\|^2_{L^{\infty}((0,T),\dot{H}^{0+} \cap \dot{H}^{1+})}
\|\phi\|_{H^{s,0}} 
\end{align*}
and in the case $ s > 1$ :
\begin{align*}
\|A_0^2 \phi \|_{H^{s-1,0}} & \lesssim \|\Lambda^{s-1}(A_0^2) \phi\|_{L^2_{xt}}
+ \|A_0^2 \Lambda^{s-1} \phi\|_{L^2_{xt}} \\
& \lesssim\|\Lambda^{s-1}(A_0^2)\|_{L^{\infty}_t L^2_x} \|\phi\|_{L^2_t
L^{\infty}_x} + \|A_0^2\|_{L^{\infty}_t L^{2+}_x} \|\phi\|_{L^2_t
H^{s-1,\infty-}_x} \\
& \lesssim \|A_0\|_{L^{\infty}_t L^{\infty-}_x}
\|\Lambda^{s-1}A_0\|_{L^{\infty}_t L^{2+}_x} \|\phi\|_{L^2_t H^s_x} +
\|A_0\|_{L^{\infty}_t \dot{H}^{0+}}^2 \|\phi\|_{L^2_t H^s_x}  \\
& \lesssim \|A_0\|^2_{L^{\infty}((0,T),\dot{H}^{0+} \cap \dot{H}^{s-1+})}
\|\phi\|_{H^{s,0}} \, . 
\end{align*}
These bounds are more than sufficient for our claim.
\end{proof}

\begin{lemma}
\label{Lemma4''}
Assume $ s > \frac{1}{4} $ , $ l > \frac{1}{4}$ and $2l-s > -\frac{1}{4}$ . Then
the following estimate applies:
$$ \|A_j A^j \phi\|_{H^{s-1,-\half++}} \lesssim \|A_j\|^2_{H^{l,\half+}}
\|\phi\|_{H^{s,\half+}} \, . $$
\end{lemma}
\begin{proof}
1. Let $s \le \frac{3}{4}$ . We use Prop. \ref{Prop.7.1} twice and obtain
\begin{align*}
\|A_j A^j \phi\|_{H^{s-1,-\half++}}  \lesssim \|A_j A^j\|_{H^{-\frac{1}{4}+,0}}
\|\phi\|_{H^{s,\half +}}  \lesssim \|A_j\|^2_{H^{l,\half+}}
\|\phi\|_{H^{s,\half+}} \, . 
\end{align*}
The conditions in Prop. \ref{Prop.7.1} are easily checked for the first
estimate, whereas the second estimate requires $2l+\frac{1}{4} >\frac{3}{4} \,
\Leftrightarrow \, l > \frac{1}{4}$ and $2l+\frac{1}{4} > 1-2l$ , which is also
fulfilled. \\
2. If $s > \frac{3}{4}$ we obtain 
\begin{align*}
\|A_j A^j \phi\|_{H^{s-1,-\half++}}  \lesssim \|A_j A^j\|_{H^{s-1,0}}
\|\phi\|_{H^{s,\half +}}  \lesssim \|A_j\|^2_{H^{l,\half+}}
\|\phi\|_{H^{s,\half+}} \, , 
\end{align*}
 where the first estimate requires $s > \frac{3}{4}$ and the last estimate $2l-s
> -\frac{1}{4}$ and $2l-s+1 > 1-2l \, \Leftrightarrow \, 4l-s >0$ , which is
true under our assumptions.
\end{proof}

\begin{lemma}
If $s > \frac{1}{4} $ , $l > \frac{1}{4}$ and $2s-l > - \frac{1}{4}$ the
following estimate applies:
$$ \|A_j |\phi|^2\|_{H^{l-1,-\frac{1}{4}++}} \lesssim
\|A_j\|_{H^{l,\frac{3}{4}+}} \|\phi\|_{H^{s,\half+}}^2 \, . $$
\end{lemma}
\begin{proof}
1. Let $l \le \frac{3}{4}$ .
We use Prop. \ref{Prop.7.1} twice and obtain
$$ \|A_j |\phi|^2\|_{H^{l-1,-\frac{1}{4}++}} \lesssim
\|A_j\|_{H^{l,\frac{3}{4}+}} \| |\phi|^2\|_{H^{-\frac{1}{4},0}}
 \lesssim \|A_j\|_{H^{l,\frac{3}{4}+}} \|\phi\|_{H^{s,\half+}}^2 \, . $$
For the first estimate we require  $\frac{3}{4} > 1-(\frac{1}{4}+l-\frac{1}{4})
\, \Leftrightarrow \, l > \frac{1}{4}$ and for the last estimate $2s+\frac{1}{4}
>\frac{3}{4} \, \Leftrightarrow \, s > \frac{1}{4}$ , which also implies
$2s+\frac{1}{4} > 1-2s$ . \\
2. Let $ l > \frac{3}{4}$ . Similarly we obtain
$$ \|A_j |\phi|^2\|_{H^{l-1,-\frac{1}{4}++}} \lesssim
\|A_j\|_{H^{l,\frac{3}{4}+}} \| |\phi|^2\|_{H^{l-1,0}}
 \lesssim \|A_j\|_{H^{l,\frac{3}{4}+}} \|\phi\|_{H^{s,\half+}}^2 \, . $$
Prop. \ref{Prop.7.1} implies both estimates, where the last estimate requires 
$2s-l > - \frac{1}{4}$.
\end{proof}

Hereafter we interpolate between the obtained bi- and trilinear estimates for
$r=1+$ and for $r=2$ .

\begin{lemma}
\label{Lemma1'}
Let $1 < r \le 2$ , $s \le l+1$ , $s > \frac{1}{r}$ and $l > \frac{3}{2r}-\half$
. The following estimate applies:
 $$ \|Q_{12}(\phi,D^{-1}A)\|_{H^r_{s-1,\frac{1}{r}-1++}} \lesssim \|A\|_{X^r_{l,
\half +\frac{1}{2r}+}} \|\phi\|_{X^r_{s,\frac{1}{r}+}} \, .$$   
\end{lemma}
\begin{proof}  
We apply Lemma \ref{Lemma5.1} for $r=1+$ , which shows 
$$ \|Q_{12}(\phi,D^{-1}A)\|_{H^r_{s_1-1,0}} \lesssim \|A\|_{X^r_{l_1,
\frac{1}{r}+,\pm_1}} \|\phi\|_{X^r_{s_1,\frac{1}{r}+,\pm_2}}$$
provided $ 1 \le s_1 \le l_1+1 $ and $ l_1 \ge 1$ . Lemma \ref{Lemma1''} implies
$$ \|Q_{12}(\phi,D^{-1}A)\|_{H^{s-1,-\half++}} \lesssim \|A\|_{X^{l_2,
\frac{3}{4}+,\pm_1}} \|\phi\|_{X^{s_2,\frac{1}{2}+,\pm_2}} $$
provided $ \half \le s_2 \le l_2+1 $ and $ l_2 >  \frac{1}{4}$ . Bilinear
interpolation completes the proof.
\end{proof}

\begin{lemma}
\label{Lemma3a}
Let $1 < r \le 2$ , $s > \frac{3}{2r}-\frac{1}{4}$ , $l > \frac{3}{2r}-\half$ , 
$ l \le s+\frac{1}{r}-\half$ and $2s-l > \frac{3}{2r}$ . Then the following
estimate applies:
$$ \|D^{-1} Q_{12}(Re \, \phi, Im \, \phi)\|_{H^r_{l-1,-\half+\frac{1}{2r}++}}
\le \|\phi\|^2_{X^r_{s,\half+}} \, . $$
\end{lemma}
\begin{proof}
By Lemma \ref{Lemma3*} we obtain for $r=1+$ the estimate
$$ \|D^{-1} Q_{12}(Re \, \phi, Im \, \psi)\|_{H^r_{l_1-1,0}} \le
\|\phi\|_{X^r_{s_1,\frac{1}{r}+,\pm_1}} \|\psi\|_{X^r_{s_1,\frac{1}{r}+,\pm_2}}\, , $$
provided $l_1 \ge 1$ , $l_1 \le s_1+\half$ and $2s_1-l_1 > \frac{3}{2}$  so that
we may assume especially $ s_1 > \frac{5}{4} $ .
By Lemma \ref{Lemma3''} the following estimate applies:
$$ \|D^{-1} Q_{12}(Re \, \phi, Im \, \psi)\|_{H^{l_2 + 1,-\frac{1}{4}++}}
\lesssim \|\phi\|_{X^{s_2,\half+,\pm_1}} \|\psi\|_{X^{s_2,\half+,\pm_2}}\, , $$
if $l_2 \le s_2$ , $s_2 > \half$ and $4s_2-l_2 > \frac{7}{4}$ . The last
condition is strengthened by  $2s_2 - l_2 > \frac{3}{4}$ , which we assume. We
also restrict to $ l_2 > \frac{1}{4}$ , recalling that we needed this assumption
in previous lemmas already. Therefore by bilinear interpolation the result
follows as one easily checks.
\end{proof}

\begin{lemma}
\label{Lemma5'}
Let $1 < r \le 2$ . Assume $s \le l+1$ , $s > \frac{3}{2r}-\half$ , $l >
\frac{7}{4r}-\frac{5}{8} $ and $2l-s > \frac{2}{r}-\frac{5}{4}$ . Then the
following estimate applies:
$$ \|A_j A^j \phi\|_{H^r_{s-1,\frac{1}{r}-1++}} \lesssim
\|A_j\|^2_{H^r_{l,\frac{1}{r}+}} \|\phi\|_{H^r_{s,\frac{1}{r}+}} \, . $$
\end{lemma}
\begin{proof}
By Lemma \ref{Lemma4} we obtain in the case $r=1+$ :
$$ \|A_j A^j \phi\|_{H^r_{s_1-1,0}} \lesssim \|A_j\|^2_{H^r_{l_1,\frac{1}{r}+}}
\|\phi\|_{H^r_{s_1,\frac{1}{r}+}} \, , $$
provided $1 \le s_1 \le l_1+1$ , $l_1 > \frac{9}{8}$ and $2l_1-s_1 >
\frac{3}{4}$ . For $r=2$ we use Lemma \ref{Lemma4''} and obtain
$$ \|A_j A^j \phi\|_{H^{s_2-1,-\half++}} \lesssim \|A_j\|^2_{H^{l_2,\half+}}
\|\phi\|_{H^{s_2,\half+}} \, , $$
if $\frac{1}{4} < s_2 \le l_2+1$ , $l_2 > \frac{1}{4}$ and $2l_2-s_2 >
-\frac{1}{4}$ , so that by trilinear interpolation the proof is completed.
\end{proof}

\begin{lemma}
\label{Lemma4'}
Let $1 < r \le 2$ . Assume $l\le s+1$ , $s > \frac{7}{4r}-\frac {5}{8} $ , $l >
\frac{3}{2r}-\half$ and $2s-l > \frac{2}{r}-\frac{5}{4}$ . The following
estimate applies:
$$ \|A_j |\phi|^2 \|_{H^r_{l-1,\frac{1}{2r}-\half++}} \lesssim
\|A_j\|_{H^r_{l,\half+\frac{1}{2r}+}} \|\phi\|^2_{H^r_{s,\frac{1}{r}+}} \, . $$
\end{lemma}
\begin{proof}
We apply Lemma \ref{Lemma5} in the case $r=1+$ and obtain
$$ \|A_j |\phi|^2 \|_{H^r_{l_1-1,0}} \lesssim \|A_j\|_{H^r_{l_1,\frac{1}{r}+}}
\|\phi\|^2_{H^r_{s_1,\frac{1}{r}+}} \, , $$
assuming $1 \le l_1 \le s_1+1$ , $s_1 > \frac{9}{8}$ and $2s_1-l_1 >
\frac{3}{4}$ . Lemma \ref{Lemma5'} implies
$$ \|A_j |\phi|^2 \|_{H^{l_2-1,-\frac{1}{4}++}} \lesssim
\|A_j\|_{H^{l_2,\frac{3}{4}+}} \|\phi\|^2_{H^{s_2,\half+}} \, , $$
if $\frac{1}{4} < l_2 \le s_2+1$ , $s_2 > \frac{1}{4}$ and $2s_2-l_2 > -
\frac{1}{4}$ . Trilinear interpolation implies the result.
\end{proof}

\begin{lemma}
\label{Lemma8'9'}
Assume $ 1 < r \le 2$ , $s > \frac{1}{r}$ , $b > \frac{1}{r}$ . The following
estimates apply:
\begin{align*}
\|A_0 \partial_t \phi\|_{H^r_{s-1,0}} & \lesssim
\|A_0\|_{L^{\infty}((0,T),\dot{H}^{\max(s-1,1+)} \cap \dot{H}^{0+})}
\|\partial_t \phi\|_{H^r_{s-1,b}} \, , \\
\|(\partial_t A_0)  \phi\|_{H^r_{s-1,0}} & \lesssim \|\partial_t
A_0\|_{L^{\infty}((0,T),\dot{H}^{\max(s-1,0+)} \cap \dot{H}^{0+})} \|
\phi\|_{H^r_{s,b}} \, ,\\
\|A_0^2 \phi\|_{H^r_{s-1,\frac{1}{r}-1++}} & \lesssim
\|A_0\|^2_{L^{\infty}((0,T),\dot{H}^{\max(s-\half,1)+} \cap \dot{H}^{0+})}
\|\phi\|_{H^r_{s,\frac{1}{r}+}} \, ,
\end{align*}
\end{lemma}
\begin{proof}
Using  bilinear interpolation between the estimates in Lemma \ref{Lemma8} and
Lemma \ref{Lemma8''} we obtain the first inequality , interpolation between the
estimates in Lemma \ref{Lemma9} and Lemma \ref{Lemma9''} gives the second
inequality. Similarly Lemma \ref{Lemma10} and Lemma \ref{Lemma10''} imply the
last inequality.
\end{proof}

Combining Lemma \ref{Lemma8'9'}, Lemma \ref{Lemma6}, Lemma \ref{Lemma7}, and
Cor. \ref{Cor.1} we obtain
\begin{Cor}
\label{Cor.2}
There exists a polynomial bound for $\|A_0 \partial_t \phi\|_{H^r_{s-1,0}}$ in
terms of $\|\phi\|_{H^r_{s,b}}$ and $\|A_j\|_{H^r_{s-1,b}}$. 
A similar bound applies for $\|(\partial_t A_0) \phi\|_{H^r_{s-1,0}}$ and also 
$\|A_0^2  \phi\|_{H^r_{s-1,\frac{1}{r}-1+}}$.
\end{Cor}

\begin{proof}[Proof of Theorem \ref{Theorem1}]
This is an application of Theorem \ref{Theorem0.3} to   the system (\ref{20}) and (\ref{21'}). The necessary estimates were
given in Lemma \ref{Lemma1'} - Lemma \ref{Lemma4'} and Corollary \ref{Cor.2}. 
We just have to check the assumptions. The most restrictive condition on $l$ is
$l > \frac{7}{4r}-\frac{5}{8}$ (cf. Lemma \ref{Lemma5'}). Thus the condition
$2s-l > \frac{3}{2r}$ (cf. Lemma \ref{Lemma3a}) implies $s >
\frac{13}{8r}-\frac{5}{16}$. Moreover we assumed $s> \frac{3}{2r}-\frac{1}{4}$
(cf. Lemma \ref{Lemma3a}), which is weaker, and $2l-s >
\frac{2}{r}-\frac{5}{4}$. The remaining conditions are easily seen to be weaker.
\end{proof}

\section{MKG in Lorenz gauge}

We can reformulate the system (\ref{1}),(\ref{2}) in Lorenz gauge
$$\partial^{\mu} A_{\mu} = 0 $$ as follows:
$$\square A_{\mu} = \partial^{\nu} \partial_{\nu} A_{\mu} =
\partial^{\nu}(\partial^{\mu} A_{\nu} - F_{\mu \nu}) = -\partial^{\nu} F_{\mu
\nu} = - j_{\mu} \, , $$
thus (using the notation $\partial = (\partial_0,\partial_1,...,\partial_n)$):
\begin{equation}
\label{16}
\square A = -Im (\phi \overline{\partial \phi}) - A|\phi|^2 =: N(A,\phi) 
\end{equation}
and
\begin{align*}
m^2 \phi & = D_{\mu} D^{\mu} \phi = \partial_{\mu} \partial^{\mu} \phi -iA_{\mu}
\partial^{\mu} \phi -i\partial_{\mu}(A^{\mu} \phi) - A_{\mu}A^{\mu} \phi \\
& = \square \phi - 2i A^{\mu} \partial_{\mu} \phi - A_{\mu} A^{\mu} \phi \,
\end{align*}
thus
\begin{equation}
\label{17}
(\square -m^2) \phi = 2i A^{\mu} \partial_{\mu} \phi + A_{\mu} A^{\mu} \phi =:
M(A,\phi) \, .
\end{equation}

As in the case of the Coulomb gauge we may also consider the equivalent first
order (in t) system. Defining
$ A_{\pm} := \frac{1}{2}(A \pm (i\Lambda)^{-1} A_t) \, , $
so that $A=A_+ + A_-$ and $A_t = i\Lambda(A_+-A_-)$, and
$\phi_{\pm} := \frac{1}{2}(\phi \pm (i \Lambda)^{-1} \phi_t) $ , so that $\phi =
\phi_+ + \phi_-$ and $\phi_t = i \Lambda (\phi_+ - \phi_-)$. 

We transform (\ref{16}),(\ref{17}) into
\begin{align}
\label{2.1}
(i\partial_t \pm \Lambda) \phi_{\pm} & = -(\pm 2 \Lambda)^{-1} M(A,\phi)
-(m^2+1)\phi\\
\label{2.2}
(i\partial_t \pm D) A_{\pm} & = -(\pm 2 \Lambda)^{-1} N(A,\phi) - A_{\pm}\, .
\end{align}

We follow Selberg-Tesfahun \cite{ST} 
and  split the spatial part $\mathbf A=(A_1,A_2)$ of the potential into
divergence-free and curl-free parts and a smoother part:
\begin{equation}\label{SplitA}  
\mathbf A = \mathbf A^{\text{df}} + \mathbf A^{\text{cf}} + \Lambda^{-2} \mathbf
A,
\end{equation}
where
\begin{align*}
\mathbf A^{\text{df}}&=  \Lambda^{-2} (\partial_2(\partial_1 A_2-\partial_2
A_1),-\partial_1(\partial_1 A_2-\partial_2 A_1))
\\
\mathbf A^{\text{cf}}&= -\Lambda^{-2} \nabla (\nabla \cdot \mathbf A).
\end{align*}
Using \eqref{SplitA} we write
	\begin{align}
\label{Aalpha}
	A^\alpha \partial_\alpha \phi
	&=
	\left( - A_0 \partial_t \phi
	+ \mathbf A^{\text{cf}} \cdot \nabla \phi \right)
	+ \mathbf A^{\text{df}} \cdot \nabla \phi +  \Lambda^{-2} \mathbf A  \cdot
\nabla \phi
	\end{align}

We have (with $R_j := \Lambda^{-1}\partial_j$):
	\begin{align}
\nonumber
	\mathbf A^{\text{df}} \cdot \nabla \phi
	&= (\Lambda^{-2}\partial_2(\partial_1A_2-\partial_2 A_1)) \partial_1 \phi -
(\Lambda^{-2}\partial_1(\partial_1 A_2 - \partial_2 A_1)) \partial_2 \phi
	\\ \nonumber
	&=-  Q_{12}(\Lambda^{-2}(\partial_1 A_2-\partial_2 A_1),\phi)
	\\ \label{P2}
	&=
	-  Q_{12}\left(\Lambda^{-1}( R_1 A_2 -R_2 A_1), \phi\right).
	\end{align}

Next we use the Lorenz gauge, $\partial_t A_0=\nabla \cdot \mathbf A  $,  to write
	\begin{align*}
	\mathbf A^{\text{cf}} \cdot \nabla \phi&
	=-\Lambda^{-2} \partial^i(\partial_t A_0) \partial_i \phi=
	- \partial_t ( \Lambda^{-1} R^i A_0)\partial_i \phi.
	\end{align*}
	We can also write
	\begin{align*}
	A_0 \partial_t \phi 
	&=-\Lambda^{-2}  \partial_i\partial^i A_0 \partial_t \phi+ \Lambda^{-2}  A_0
\partial_t \phi
	\\
	&=-\partial_i(\Lambda^{-1} R^i A_0) \partial_t \phi
	+  \Lambda^{-2}  A_0 \partial_t \phi.
	\end{align*}
	Combining the above identities, we get
	\begin{align}
\label{A}
	 -A_0 \partial_t \phi +  \mathbf A^{\text{cf}} \cdot \nabla \phi
	&=Q_{i0}(\Lambda^{-1} R^i A_0, \phi)- \Lambda^{-2}  A_0 \partial_t \phi \, ,
	\end{align}
where we define the null form
$$Q_{i0}(u,v) = -\partial_t u \partial_i v + \partial_i u \partial_t v \, . $$
We may ignore the factor $R^i$ . By the definition of the null form we obtain with $D_i := \frac{\partial_i}{i}$ with symbol $\xi_i$ :
\begin{align*}
&Q_{i0}(\Lambda^{-1} A_0,\phi) \\
& = \sum_{\pm_1,\pm_2} (\pm_1 1)(\pm_2 1)\left(-A_{0,\pm_1} (\pm_2 D_i) \phi_{\pm_2} + (\pm_1 D_i) \Lambda^{-1} A_{0,\pm_1} \Lambda \phi_{\pm_2}\right) \, .
\end{align*}
Its symbol is given by
$$ \sum_{\pm_1,\pm_2} (\pm_1 1)(\pm_2 1) \left(-(\pm_2(\xi - \eta)_i) \pm_1 \eta_i\langle \eta \rangle^{-1} \langle \xi - \eta \rangle  \right)\, . $$
This symbol is bounded by
\begin{align}
\label{38}
&\left| - \frac{(\pm_2(\xi-\eta))_i |\eta|}{\langle \eta \rangle  } + \frac{(\pm_1 \eta)_i |\xi-\eta|}{\langle \eta \rangle} \right| + |\xi-\eta| \frac{\langle \eta \rangle - |\eta|}{\langle \eta \rangle} + |\eta| \frac{\langle \xi - \eta \rangle - |\xi-\eta|}{\langle \eta \rangle} \\
\label{Qi0}
& \lesssim \frac{ |\xi-\eta| |\eta|}{\langle \eta \rangle} \angle(\pm_1 \eta,\pm_2 (\xi-\eta)) + \frac{|\xi-\eta|}{\langle \eta \rangle} + 1 = : R_1+R_2+R_3\,.
\end{align}

If we use the estimate for the angle in Lemma \ref{Lemmaangle} below (for $\alpha,\beta,\gamma=\half$) for the term $R_1$ we obtain by (\ref{Qi0}) the following estimate:
\begin{align}
\nonumber
Q_{i0}(\Lambda^{-1}u,v)  \precsim &D^{\half}_-(\Lambda^{-\half}u Dv) + D^{\half}_-(u D^{\half}v) + (D^{\half}_- \Lambda^{-\half}u) Dv + D^{\half}_-u D^{\half}v \\
\label{40}
&+ \Lambda^{-\half} u (D^{\half}_- Dv) + u (D^{\half}_- D^{\half}v) + \Lambda^{-1}u \Lambda v + uv
\end{align}

The estimate for the angle in the following Lemma was
proven in \cite{AFS1}, Lemma 7:
\begin{lemma}
	\label{Lemmaangle}
	Let $\alpha,\beta,\gamma \in [0,\half]$ , $\tau,\lambda \in \R$ , $ \xi,\eta
\in \R^2 $, $\xi,\eta \neq 0$ . Then the following estimate applies for all
signs $\pm_1,\pm_2$ :
$$\angle{(\pm_1 \xi, \pm_2 \eta)} \lesssim
\left(\frac{\langle|\tau+\lambda|-|\xi+\eta|\rangle}{\min(\langle\xi
\rangle,\langle \eta \rangle)}\right)^{\alpha} +\left(\frac{\langle -\tau
\pm_1|\xi|\rangle}{\min(\langle\xi\rangle,\langle \eta \rangle)}\right)^{\beta}
+\left(\frac{\langle -\lambda \pm_2|\eta|\rangle}{\min(\langle\xi
\rangle,\langle \eta \rangle)}\right)^{\gamma} \, .
 $$
\end{lemma}

Similarly as in  Lemma \ref{Lemma5.1}  we can also estimate the (modified) nullform $q_{0j}(u,v)$ .
\begin{lemma}
\label{Lemma5.2}
Assume $0 \le \alpha_1,\alpha_2 $ ,  $\alpha_1+\alpha_2 \ge \frac{1}{r}$ and $ b > \frac{1}{r}$ . The following estimate applies
$$ \|q_{0j}^{\pm}(u,v)\|_{H^r_{0,0}} \lesssim \|u\|_{X^r_{\alpha_1,b,\pm_1}} \|v\|_{X^r_{\alpha_2,b,\pm_2}} \, . $$
where
$$q_{0j}^{\pm}(u,v) := -u (D^{-1}\partial_j v) \pm (D^{-1}\partial_j u)  v \, . $$
\end{lemma}
\begin{proof}
	Again we may  reduce to the case $\alpha_1= \frac{1}{r}$ and $\alpha_2=0$ .
Arguing as in the proof of Lemma \ref{Lemma5.1} we use in the elliptic case the estimate (cf. \cite{FK}, Lemma 13.2):
$$|q_{0j}^+(\eta,\xi-\eta)| \lesssim \frac{(|\eta|+|\xi-\eta|-|\xi|)^{\half}}{\min(|\eta|^{\half},|\xi-\eta|^{\half})} \, . $$
In the case  $|\eta| \le |\xi-\eta|$ we obtain
\begin{align*}
I &= ||\tau|-|\xi||^{\half} \left( \int \delta(\tau - |\eta| - |\xi - \eta|) \, |\eta|^{-1-\frac{r}{2}} d\eta \right)^{\frac{1}{r}} \\
&\sim ||\tau|-|\xi||^{\half} |\tau|^{\frac{A}{r}}  ||\tau|-|\xi||^{\frac{B}{r}} = 1\, ,
\end{align*}
because $ A=\max(1+\frac{r}{2},\frac{3}{2}) - 1-\frac{r}{2} = 0$ and $B= 1-\max(1+\frac{r}{2},\frac{3}{2})-\frac{r}{2} = -\frac{r}{2}$ . \\
In the case $|\eta| \ge |\xi-\eta|$ we obtain
\begin{align*}
I &= ||\tau|-|\xi||^{\half} \left( \int \delta(\tau - |\eta| - |\xi - \eta|) \, |\eta|^{-1} |\xi-\eta|^{-\frac{r}{2}} d\eta \right)^{\frac{1}{r}} \\
&\sim ||\tau|-|\xi||^{\half} |\tau|^{\frac{A}{r}}  ||\tau|-|\xi||^{\frac{B}{r}} (1+\log \frac{|\tau|}{||\tau|-|\xi||})^{\frac{1}{r}}\, ,
\end{align*}
where  $A= \max(1,\frac{r}{2},\frac{3}{2})-1-\frac{r}{2} = \half-\frac{r}{2}$ and $B=-\half$, so that
$$I \lesssim ||\tau|-|\xi||^{\half} \tau^{\frac{1}{2r}-\half} ||\tau|-|\xi||^{-\frac  {1}{2r}} \lesssim  1 \, .$$
 In the hyperbolic case we obtain by \cite{FK}, Lemma 13.2:
$$|q_{0j}^-(\eta, \xi-\eta)| \lesssim |\xi|^{\half} \frac{(|\xi|-||\eta|-|\eta-\xi||)^{\half}}{|\eta|^{\half} |\xi-\eta|^{\half}} $$
and argue exactly as in the proof of Lemma \ref{Lemma5.1}. The proof is completed as before.
\end{proof}

\begin{Cor}
\label{Cor.3}
Let $1 < r \le 2$ .
Assume $s\le l+1$ , $l \ge \frac{1}{r}$ , $s \ge 1$ , $ b > \frac{1}{r}$ . The following estimate applies
$$\|q_{0j}(u,Dv)\|_{H^r_{s-1,0}}\lesssim  \|u\|_{X^r_{l,b,\pm_1}} \|v\|_{X^r_{s,b,\pm_2}} \, . $$
\end{Cor}	 

The last term in (\ref{A}) and $R_1$ and $R_2$ in (\ref{41}) are easily  treated by Lemma \ref{Lemma3'} in the following lemma.
\begin{lemma}
\label{Lemma4.2} 
Let $1 < r \le 2$ , $s \ge 0$ , $l> \frac{1}{r}-\frac{3}{2}$ , $ s \le l+1$ and $b> \frac{1}{r}$ . Then the following estimates apply:
\begin{align}
\label{41}
 \| \Lambda^{-1}u \Lambda v\|_{H^r_{s-1,0}} &\lesssim \|u\|_{H^r_{l,b}} \|\Lambda v\|_{H^r_{s-1,b}} \, , \\
\| u v\|_{H^r_{s-1,0}} &\lesssim \|u\|_{H^r_{l,b}} \|\Lambda v\|_{H^r_{s-1,b}} 
\end{align}
\end{lemma}
\begin{proof}
By Lemma \ref{Lemma3'} this is satisfied, if  $1 < r \le 2$ , $ s \ge 1$ , $l \ge \frac{3}{2r}-1 $ , $s-1 \le l$ and $b> \frac{1}{r}$ .
We use this result for $r=1+$ .
By Sobolev this estimate is certainly true in the case $r=2$ , $b=\half+$ , if $l>0$ and $s-1 \le l$ . By interpolation between the cases $r=1+$ and $r=2$ we obtain the result.
\end{proof}

Next we estimate the null form $Q_{0j}$  in the case $r=2$ .
\begin{lemma}
\label{Lemma2'} Assume $ s > \half$ , $ l >\frac{1}{4}$ and $2l-s > -\half$ .
Then the following estimate applies:
$$\| Q_{0j}(\Lambda^{-1}u,v)\|_{H^{s-1,-\half++}} \lesssim \|u\|_{X^{l,\frac{3}{4},\pm_1}}
\|Dv\|_{X^{s-1,\half+,\pm_2}} \, $$
\end{lemma}
\begin{proof}
 We use the estimate (\ref{40}). This reduces the claimed estimate to the following eight
inequalities:
\begin{align}
\label{1a}
\|uv\|_{H^{s-1,0+}} & \lesssim \|u\|_{H^{l+\half-,\frac{3}{4}+}}
\|v\|_{H^{s-1,\half+}} \\
\label{1b}
\|uv\|_{H^{s-1,0+}} & \lesssim \|u\|_{H^{l,\frac{3}{4}+}}
\|v\|_{H^{s-\half-,\half+}} \\
\label{1c}
\|uv\|_{H^{s-1,-\half+}} & \lesssim \|u\|_{H^{l+\half-,\frac{1}{4}+}}
\|v\|_{H^{s-1,\half+}} \\
\label{1d}
\|uv\|_{H^{s-1,-\half+}} & \lesssim \|u\|_{H^{l,\frac{1}{4}+}}
\|v\|_{H^{s-\half-,\half+}} \\
\label{1e}
\|uv\|_{H^{s-1,-\half+}} & \lesssim \|u\|_{H^{l+\half-,\frac{3}{4}+}}
\|v\|_{H^{s-1,0}} \\
\label{1f}
\|uv\|_{H^{s-1,-\half+}} & \lesssim \|u\|_{H^{l,\frac{3}{4}+}}
\|v\|_{H^{s-\half-,0}} \\
\label{1g} 
\|uv\|_{H^{s-1,-\half+}} & \lesssim \|u\|_{H^{l+1,\frac{3}{4}+}}
\|v\|_{H^{s-1,\half+}} \\
\label{1h}
\|uv\|_{H^{s-1,-\half+}} & \lesssim \|u\|_{H^{l,\frac{3}{4}+}}
\|v\|_{H^{s,\half+}} \, .
\end{align}
(\ref{1g}) and (\ref{1h}) are easily proven by Sobolev. For the other estimates we apply Prop. \ref{Prop.7.1}. Using the notation in this proposition we require
in all cases:
$s_0+s_1+s_2 = l+\half-  > \frac{3}{4} \, \Leftrightarrow \, l > \frac{1}{4}$
.\\
(\ref{1a}): $l+\half > 1-(l+\half)-(s-1) \, \Leftrightarrow \, 2l+s > 1$ , which
is true in our case $s > \half$ and $l > \frac{1}{4}$ . \\
(\ref{1b}) is true if $2l+s > 1$ . \\
(\ref{1c}): This is the point where it is important to consider $u \in H^{l,b}$
with $ b > \frac{3}{4}$  instead of the standard choice $b=\half+$ . Because of
$b_1= \frac{1}{4}+$ we only need $l+\half > 1-(1-s)-(s-1)-\frac{1}{4} \,
\Leftrightarrow \, l > \frac{1}{4}$ (instead of $l > \half$).\\
(\ref{1d}) and (\ref{1e}) require no additional assumptions. \\
(\ref{1f}): This requires the condition $l+\half > 1-(1-s)-l \, \Leftrightarrow
\, 2l-s > -\half$ .
\end{proof}

\begin{lemma}
\label{Lemma2''}
Let $1 < r \le 2$ , $s > \frac {1}{r}$ , $l > \frac{3}{2r}-\half$ , $2l-s > 
\frac{1}{r}-1$ . Then the following estimate applies:
$$ \|Q_{0j}(\Lambda^{-1}u,v)\|_{H^r_{s-1,\frac{1}{r}-1++}} \lesssim
\|u\|_{X^r_{l,\half+\frac{1}{2r}+}} \|v\|_{X^r_{s,\frac{1}{r}+}} \, . $$
\end{lemma}
\begin{proof}
We remark, that by (\ref{38}) we haven proven that
$$ Q_{0j}(\Lambda^{-1}u,v) \precsim \sum_{+,-} q_{0j}^{\pm}(u,Dv)+ \Lambda^{-1} u Dv + uv \, . $$
We interpolate between Cor. \ref{Cor.3} (and Lemma \ref{Lemma4.2}) for the case $r=1+$ , which requires $ s
\ge 1$ and $l \ge 1$ as well as $s-1 \le l$ , which implies $2l-s > 2l-(l+1)=
l-1 \ge 0$ , and Lemma \ref{Lemma2'} for $r=2$ , which requires $ s > \half$ ,
$l > \frac{1}{4}$ and $2l-s >-\half$. This completes the proof, as one easily
checks.
\end{proof}

\begin{lemma}
\label{Lemma3}
Let $1<r \le 2$ , $ b > \frac{1}{r}$ , $s\ge l$ , $l \ge 1$ and $2s-l >
\frac{3}{2r}$ .Then
$$ \| Im(\phi \overline{\partial \phi})\|_{H^r_{l-1}} \lesssim
\|\phi\|_{H^r_{s,b}} ^2 \, .$$
\end{lemma}
\begin{proof}
An application of  the fractional Leibniz rule and Lemma \ref{Lemma3'} implies
the result.
\end{proof}

\begin{lemma}
\label{Lemma3b}
Assume $s \ge l$ , $ s \ge \half$ and $2s-l > \frac{3}{4}$ . The following
estimate applies:
$$ \| Im(\phi \overline{\partial \phi})\|_{H^{l-1,-\frac{1}{4}++}} \lesssim
\|\phi\|_{H^{s,\half+}} \|\partial \phi\|_{H^{s-1,\half+}} \, . $$
\end{lemma}
\begin{proof}
We use Prop. \ref{Prop.7.1}, which requires $2s-l > \frac{3}{4}$ and $s \ge
\half$ . This implies also  $4s-l > \frac{7}{4}$ , which is also required. 
\end{proof}

\begin{lemma}
\label{Lemma3''''}
Let $1 <  r \le 2$ , $s \ge l$ , $l > \frac{3}{2r}-\half$ , $2s-l >
\frac{3}{2r}$ . The following estimate applies:
$$ \|Im(\phi \overline{\partial \phi})\|_{H^r_{l-1,\frac{1}{2r}-\half++}}
\lesssim \|\phi\|_{H^r_{s,\frac{1}{r}+}} \|\partial
\phi\|_{H^r_{s-1,\frac{1}{r}+}} \, . $$
\end{lemma}
\noindent Remark: The assumptions of the lemma imply $ s > \frac{3}{2r} -
\frac{1}{4}$ .
\begin{proof}
By Lemma \ref{Lemma3} for $r=1+$ we obtain
$$ \|Im(\phi \overline{\partial \phi})\|_{H^r_{l_1-1,0}} \lesssim
\|\phi\|_{H^r_{s_1,\frac{1}{r}+}} \|\partial \phi\|_{H^r_{s_1-1,\frac{1}{r}+}}
$$
for $s_1 \ge l_1 \ge 1$ , $2s_1-l_1 > \frac{3}{2}$ , which implies $s_1 >
\frac{5}{4}$, and Lemma \ref{Lemma3b} implies
$$ \|Im(\phi \overline{\partial \phi})\|_{H^{l_2-1,-\frac{1}{4}++}} \lesssim
\|\phi\|_{H^{s_2,\frac{1}{2}+}} \|\partial \phi\|_{H^{s_2-1,\frac{1}{2}+}} \, ,
$$ 
provided $s_2\ge l_2$ and $2s_2-l_2 > \frac{3}{4}$ . We may also assume $l_2 >
\frac{1}{4}$ , because this condition is required for other estimates already,
which implies automatically $s_2 > \half$ . Bilinear interpolation between these
inequalities implies the claimed result.
\end{proof}

\begin{proof}[Proof of Theorem \ref{Theorem2}]
We apply Theorem \ref{Theorem0.3} to the system (\ref{2.1}) and (\ref{2.2}). Using (\ref{Aalpha}),(\ref{P2}) and (\ref{A}) the necessary estimates for the term $A^{\mu}
\partial_{\mu} \phi$ were given in Lemma \ref{Lemma1'}, Lemma \ref{Lemma2''} and
Lemma \ref{Lemma4.2}. Moreover we need Lemma \ref{Lemma3''''}, Lemma
\ref{Lemma4'} and Lemma \ref{Lemma5'}.  We just have to check the assumptions.
The most restrictive condition on $l$ is $l > \frac{7}{4r}-\frac{5}{8}$ (cf.
Lemma \ref{Lemma5'}). Thus the condition $2s-l > \frac{3}{2r}$ (cf.
Lemma \ref{Lemma3''''}) is only compatible, if $s > \frac{13}{8r}-\frac{5}{16}$.
Moreover we assumed  $2l-s > \frac{2}{r}-\frac{5}{4}$. The remaining conditions
are weaker.
\end{proof}

\end{document}